\documentclass[a4paper]{amsart}

\usepackage[utf8]{inputenc}
\usepackage[T1]{fontenc}
\usepackage{lmodern, enumerate}
\usepackage{amssymb,amsxtra}
\usepackage{amscd}
\usepackage[all]{xy}
\usepackage{xcolor}
\usepackage{nicefrac,mathtools}
\usepackage{microtype}
\usepackage[pdftitle={Universal and exotic generalized fixed-point algebras},
 pdfauthor={Alcides Buss, Siegfried Echterhoff},
 pdfsubject={Mathematics}
]{hyperref}
\usepackage[lite]{amsrefs}
\newcommand*{\MRref}[2]{ \href{http://www.ams.org/mathscinet-getitem?mr=#1}{MR #1}}
\newcommand*{\arxiv}[1]{\href{http://www.arxiv.org/abs/#1}{arXiv: #1}}

\numberwithin{equation}{section}
\theoremstyle{plain}
\newtheorem{theorem}[equation]{Theorem}
\newtheorem{lemma}[equation]{Lemma}
\newtheorem{proposition}[equation]{Proposition}
\newtheorem{corollary}[equation]{Corollary}
\theoremstyle{definition}
\newtheorem{definition}[equation]{Definition}
\theoremstyle{remark}
\newtheorem{remark}[equation]{Remark}
\newtheorem{example}[equation]{Example}

\DeclareMathOperator{\Aut}{Aut}
\DeclareMathOperator{\cspn}{\overline{span}}
\DeclareMathOperator{\Ind}{Ind}
\DeclareMathOperator{\supp}{\mathrm{supp}}
\DeclareMathOperator{\id}{\mathrm{id}}
\DeclareMathOperator{\triv}{\mathrm{tr}}
\DeclareMathOperator{\Inf}{Inf}
\DeclareMathOperator{\im}{Im}
\DeclareMathOperator{\Mor}{\mathrm{Mor}}
\DeclareMathOperator{\pt}{\mathrm{pt}}

\newcommand*{\CP}{\mathrm{CP}} 

\newcommand*{\nb}{\nobreakdash}
\newcommand*{\Star}{\(^*\)\nobreakdash-}
\newcommand*{\dd}{\,d}

\newcommand*{\C}{\mathbb C}
\newcommand*{\Lb}{\mathcal L}
\renewcommand*{\L}{\mathcal L}
\newcommand*{\K}{\mathcal K}
\renewcommand*{\H}{\mathcal H}
\newcommand*{\F}{\mathcal F} 
\newcommand*{\Cst}{C^*}
\newcommand*{\cont}{C}
\newcommand*{\contz}{\cont_0}
\newcommand*{\contc}{\cont_c}
\newcommand*{\contb}{\cont_b}

\newcommand*{\M}{\mathcal M}

\newcommand*{\Id}{\textup{id}}
\newcommand*{\Ad}{\textup{Ad}}

\newcommand*{\U}{\mathcal U}
\newcommand*{\E}{\mathcal E}
\newcommand*{\EE}{\mathbb E}
\newcommand*{\En}{E} 
\newcommand*{\Y}{\mathcal Y}%

\newcommand*{\defeq}{\mathrel{\vcentcolon=}}
\newcommand*{\congto}{\xrightarrow\sim}

\newcommand*{\braket}[2]{\langle#1\!\mid\!#2\rangle}
\newcommand*{\Braket}[2]{\left\langle#1\!\mid\!#2\right\rangle}
\newcommand*{\bbraket}[2]{\mathopen{\langle\!\langle}#1\!\mid\!#2\mathclose{\rangle\!\rangle}}

\renewcommand*{\r}{\textup r}
\newcommand*{\sbe}{\subseteq} 

\newcommand*{\cstar}{\texorpdfstring{$C^*$\nobreakdash-\hspace{0pt}}{*-}}
\newcommand*{\into}{\hookrightarrow}
\newcommand*{\onto}{\twoheadrightarrow}
\newcommand*{\red}{r}
\renewcommand*{\max}{\mathrm{max}}
\newcommand*{\un}{u}
\newcommand*{\pn}{\mathrm{\mu}}
\newcommand*{\qn}{\mathrm{\nu}}
\newcommand*{\Fix}{\mathrm{Fix}}
\newcommand*{\dual}[1]{\widehat{#1}}
\newcommand*{\dualG}{\widehat{G}}
\newcommand*{\I}{\mathcal{I}} 
\newcommand{\Zt}{\mathcal{Z}} 
\newcommand{\point}{\mathrm{pt}}

\newcommand*{\st}{\mathrm{st}}
\newcommand*{\bs}{\backslash}

\newcommand{\ie}{\emph{i.e.}}
\newcommand{\eg}{\emph{e.g.~}}
\newcommand{\cf}{\emph{cf.}}
\renewcommand{\iff}{if and only if }

\newcommand*{\Cstd}[2]{\mathfrak C^*({#1},{#2})}
\newcommand*{\Cstnd}[2]{\mathfrak C^*_{\mathrm{nd}}({#1},{#2})}
\newcommand*{\CCstd}[1]{\mathfrak C^*({#1})}
\newcommand*{\CCstnd}[1]{\mathfrak C^*_{\mathrm{nd}}({#1})}
\newcommand*{\CCCstd}{\mathfrak C^*}
\newcommand*{\CCCstnd}{\mathfrak C^*_{\mathrm{nd}}}

\begin{document}
\title[Generalized fixed-point algebras]{Universal and exotic generalized fixed-point algebras for weakly proper actions and duality}

\author{Alcides Buss}
\email{alcides.buss@ufsc.br}
\address{Departamento de Matem\'atica\\
 Universidade Federal de Santa Catarina\\
 88.040-900 Florian\'opolis-SC\\
 Brazil}

\author{Siegfried Echterhoff}
\email{echters@math.uni-muenster.de}
\address{Mathematisches Institut\\
 Westf\"alische Wilhelms-Universit\"at M\"un\-ster\\
 Einsteinstr.\ 62\\
 48149 M\"unster\\
 Germany}

\begin{abstract}
Given  a \cstar{}dynamical system $(A,G,\alpha)$, we say that $A$ is a weakly proper
$X\rtimes G$-algebra if there exists a proper $G$\nb-space $X$ together with a
nondegenerate $G$\nb-equivariant \Star{}homomorphism $\phi:\contz(X)\to \M(A)$.
 Weakly proper $G$\nb-algebras form a large subclass of the class of
 proper $G$\nb-algebras in the sense of Rieffel.
 In this paper we show that weakly proper $X\rtimes G$-algebras allow the
 construction of \emph{full} fixed-point algebras $A^{G,\alpha}_u$
 corresponding to the full crossed product $A\rtimes_{\alpha}G$, thus solving, in this setting,
 a problem stated by Rieffel in his 1988's original article on proper actions.
 As an application we obtain a general Landstad duality result for arbitrary coactions
 together with a new and functorial construction of maximalizations of coactions.

 The same methods also allow the construction of exotic generalized fixed-point algebras associated to crossed-product norms lying between the reduced and universal ones.
 Using these, we give complete answers to some questions on duality theory for exotic crossed products recently
 raised by Kaliszewski, Landstad and Quigg. 
\end{abstract}



\thanks{Supported by Deutsche Forschungsgemeinschaft  (SFB 878, Groups, Geometry \& Actions), CNPq (Ciências sem Fronteira) {and PVE/CAPES} -- Brazil. }
\thanks{We would like to thank the referee for the valuable suggestions which helped us to improve this paper.}

\maketitle

\section{Introduction}\label{sec:introduction}
If a locally compact group $G$ acts properly on a locally compact space $X$, then
we call a \cstar{}algebra $A$ a \emph{weakly proper $X\rtimes G$-algebra} if there exists an
action $\alpha:G\to \Aut(A)$ together with a $G$\nb-equivariant nondegenerate \Star{}homomorphism
$\phi:\contz(X)\to \M(A)$. The notion of weakly proper $X\rtimes G$-algebras is much weaker
than the notion of an action of the proper groupoid $X\rtimes G$, which requires that
$\phi$ takes values in the center $\Zt\M(A)$ of the multiplier algebra $\M(A)$.
On the other hand, the notion of weakly proper $X\rtimes G$-algebras is stronger than
Rieffel's notion of proper $G$\nb-algebras (or proper actions on \cstar{}algebras) as
introduced in \cite{Rieffel:Proper}.
In fact, as shown by Rieffel in \cite[Proposition 5.7]{Rieffel:Integrable_proper}, the dense subalgebra
$A_c:=\phi(\contc(X))A\phi(\contc(X))$ carries a natural $A\rtimes_{\alpha,r}G$-valued inner product
and the (reduced) generalized fixed-point algebra $A_r^{G,\alpha}$ is just
the algebra of compact operators of the module $\F(A)_r$ which denotes the completion
of $A_c$ as a Hilbert $A\rtimes_{\alpha,r}G$-module.
Rieffel's generalized fixed-point algebra can be realized as
the closure in $\M(A)$ of the \emph{fixed-point algebra with compact supports}
\begin{equation}\label{eq-fix-c}
A_c^{G,\alpha}:= \phi(\contc(G\backslash X)) \widetilde{A}_c^{G,\alpha}\phi(\contc(G\backslash X))
\end{equation}
with
$$\widetilde{A}_c^{G,\alpha}=\{m\in \M(A)^{G,\alpha}: m\phi(f), \phi(f)m\in A_c\,\mbox{ for all } f\in \contc(X)\}$$
where $\M(A)^{G,\alpha}$ denotes the classical fixed-point algebra in $\M(A)$ under the extended action.

Looking at these results, it seems at first sight that the theory of generalized fixed-point algebras in relation
to the crossed product of $A$ with $G$ makes only sense for reduced crossed products.
In fact, as observed by Rieffel in \cite[pp. 145-146]{Rieffel:Proper}, it is not clear for general
(Rieffel-) proper actions, whether the $L^1(G,A)$-valued inner product on $A_c$ satisfies the
positivity condition for \cstar{}valued inner products when mapped into the full crossed product $A\rtimes_{\alpha}G$.
On the other hand, in the recent paper \cite{Kaliszewski-Muhly-Quigg-Williams:Fell_bundles_and_imprimitivity_theoremsII}, it is shown that for certain proper actions
on \cstar{}algebras which can be realized as cross-sectional algebras of Fell bundles over certain groupoids, the
inner product with values in the maximal crossed product does make sense, and
it leads to a universal version of Rieffel's generalized fixed-point algebra in these cases.
It lies in the nature of cross-sectional algebras of Fell bundles that the results in \cite{Kaliszewski-Muhly-Quigg-Williams:Fell_bundles_and_imprimitivity_theoremsII}
are technically quite challenging. Also, one observes that
all examples considered in \cite{Kaliszewski-Muhly-Quigg-Williams:Fell_bundles_and_imprimitivity_theoremsII} are weakly proper $X\rtimes G$-algebras
in our sense for suitable $X$ (see Step~1 in the proof of \cite[Proposition~3.3]{Kaliszewski-Muhly-Quigg-Williams:Fell_bundles_and_imprimitivity_theoremsII}).

In this paper we show that, indeed, for every weakly proper $X\rtimes G$-algebra $A$,
the canonical $\contc(G,A)$-valued inner product on $A_c$ gives a well-defined inner product
with values in $A\rtimes_{\alpha}G$ and hence $A_c$ completes to
a Hilbert $A\rtimes_{\alpha}G$-module $\F_\un(A)$.
The \cstar{}algebra of compact operators on $\F_\un(A)$
is realized as a completion $A_u^{G,\alpha}$
of the generalized fixed-point algebra with compact supports $A_c^{G,\alpha}$ of (\ref{eq-fix-c}) and we
 call $A_u^{G,\alpha}$ the \emph{universal generalized fixed-point algebra of $A$}.
We say that $A$ is \emph{saturated}, if the $A\rtimes_{\alpha}G$-valued inner
product on $\F_\un(A)$ is full, in which case $\F_\un(A)$ becomes a
$A_u^{G,\alpha}-A\rtimes_{\alpha}G$ equivalence bimodule.
This is always the case if the action of $G$ on $X$ is \emph{free} and proper, but this
is not a necessary condition.
In general $\F_\un(A)$ will be a \emph{partial} $A_u^{G,\alpha}-A\rtimes_{\alpha}G$
equivalence bimodule implementing a Morita equivalence between
$A_u^{G,\alpha}$ and a suitable ideal in $A\rtimes_{\alpha}G$.

Hence the universal theory of generalized fixed-point algebras always works in the case of weakly proper actions.
Although a lot of progress has been done recently in the setting of generalized fixed-point algebras for weakly proper actions (\eg see \cite{Huef-Kaliszewski-Raeburn-Williams:Naturality_Rieffel, Huef-Raeburn-Williams:FunctorialityGFPA, Kaliszewski-Quigg-Raeburn:ProperActionsDuality}),
it is quite surprising that this fact has not been noticed so far.

Our methods also allow the construction of generalized fixed-point algebras associated to \emph{exotic} crossed products, meaning \cstar{}completions of $\contc(G,A)$ lying between
$A\rtimes_\un G$ and $A\rtimes_\red G$, as discussed in a recent paper
by Kaliszewski, Landstad and Quigg \cite{Kaliszewski-Landstad-Quigg:Exotic} and motivated by an earlier study of
exotic group algebras by Brown and Guentner \cite{Brown-Guentner:New_completions}. In the second part of this article we use this kind of generalized fixed-point algebra for weakly proper
$G\rtimes G$\nb-algebras (in which $G$ acts on itself by right translation) to study
duality results for exotic crossed products. In \cite{Kaliszewski-Landstad-Quigg:Exotic} it is shown that any $G$\nb-invariant
weak*-closed ideal $E$ in the Fourier-Stieltjes
algebra $B(G)$ of $G$ gives rise to exotic crossed-product norms $\|\cdot\|_E$ such that the
crossed products $A\rtimes_{\alpha,E}G$ always admit dual coactions $\widehat{\alpha}_E$.
Using such norms, the $A_u^{G,\alpha}-A\rtimes_{\alpha}G$ bimodule $\F_\un(A)$
factors to give a partial equivalence $A^{G,\alpha}_E-A\rtimes_{\alpha,E}G$ bimodule
$\F(A)_E$ and following ideas of Quigg in \cite{Quigg:Landstad_duality},
we shall show that there exists a canonical coaction of $G$ on $\F(A)_E$ which is compatible
with $\widehat{\alpha}_E$ and therefore implements a coaction $\delta_E$ on
the \emph{$E$\nb-generalized fixed-point algebra $A^{G,\alpha}_E$}.
The dual system for the
coaction crossed product $(A_E^{G,\alpha}\rtimes_{\delta_E}\widehat{G}, \widehat{\delta}_E)$ is isomorphic to the
original system $(A,\alpha)$ and the coaction $\delta_E$ on $B_E:=A_E^{G,\alpha}$
satisfies \emph{$E$\nb-Katayama duality} in the sense that
$$B_E\rtimes_{\delta_E}G\rtimes_{\widehat{\delta},E}G\cong B_E\otimes \K(L^2G)$$
via a certain canonical map. (In fact a similar statement is true for any
completion $A\rtimes_{\alpha,\pn}G$ of $\contc(G,A)$ by any \cstar{}norm $\|\cdot\|_\pn$ which admits a
dual coaction $\widehat{\alpha}_\pn$. Such norms are not necessarily attached to an ideal $E$ in $B(G)$).
Thus our results can be viewed as a version of Landstad duality for $E$\nb-coactions.

In particular, if we start with any coaction $\delta:B\to \M(B\otimes \Cst(G))$, then $A=B\rtimes_{\delta}\widehat{G}$
becomes a weakly proper $G\rtimes G$\nb-algebra in a canonical way. Given $E\subseteq B(G)$ as above,
the corresponding coaction $\delta_E$ on $B_E:=(B\rtimes_{\delta}\widehat{G})^{G,\widehat{\delta}}_E$
has the following properties:
\begin{itemize}
\item[(E1)] $\delta_E$ satisfies $E$\nb-Katayama duality, \ie,
$B_E\rtimes_{\delta_E}G\rtimes_{\widehat{\delta},E}G\cong B_E\otimes \K(L^2G)$.
\item[(E2)] The dual systems $(B\rtimes_\delta \widehat{G}, \widehat{\delta})$ and $(B_E\rtimes_{\delta_E}G, \widehat{\delta}_E)$
are isomorphic.
\item[(E3)] If $E=B(G)$, then $(B_u, \delta_u):=(B_{B(G)},\delta_{B(G)})$ is the maximalization of $(B, \delta)$ and if $E=B_r(G)$ (the weak*-closure
of $A(G)$), then $(B_r, \delta_r):=(B_{B_r(G)},\delta_{B_r(G)})$ is the normalization of $\delta$.
\end{itemize}
Note that for $E=B(G)$ we get maximal crossed products
and for $E=B_r(G)$ we get reduced crossed products. For $(B_u, \delta_u)$ and $(B_r, \delta_r)$
being the \emph{maximalization} and \emph{normalization} of $\delta$ means in particular that
the dual systems
$$(B_u\rtimes_{\delta_u}\widehat{G},\widehat{\delta}_u), \quad (B\rtimes_{\delta}\widehat{G},\widehat{\delta})
\quad\text{and}\quad (B_r\rtimes_{\delta_r}\widehat{G},\widehat{\delta}_r)$$
coincide, $(B_u,\delta_u)$ satisfies duality for the full crossed product and
$(B_r,\delta_r)$ satisfies duality for the reduced crossed products.
Normal coactions  have been introduced by Quigg in \cite{Quigg:FullAndReducedCoactions} and they
are in natural one-to-one correspondence to coactions of the reduced group C*-algebra $C_r^*(G)$.
It is also shown in \cite{Quigg:FullAndReducedCoactions} that every coaction has a normalization as above.
A proof that maximalizations exist was first
given in \cite{Echterhoff-Kaliszewski-Quigg:Maximal_Coactions}  (see \cite{Fischer} for a construction
in case of  quantum groups), but the construction given here is canonical and has better functorial properties.
Indeed, in the two final sections of this paper we prove that our constructions of exotic (and, in particular, maximal and reduced) generalized fixed-point algebras
are functorial for categories based on (equivariant) homomorphisms between \cstar{}algebras and extend to exotic norms the categorical version of Landstad duality
obtained by Kaliszewski, Quigg and Raeburn in \cite{Kaliszewski-Quigg-Raeburn:ProperActionsDuality}.

Furthermore, from the above results we also deduce a positive answer to \cite[Conjecture 6.14]{Kaliszewski-Landstad-Quigg:Exotic}:
if $E$ is a $G$\nb-invariant weak*-closed ideal in $B(G)$, then for any
action $\alpha:G\to \Aut(A)$ the dual coaction $\widehat{\alpha}_E$ on $A\rtimes_{\alpha,E}G$
satisfies $E$\nb-Katayama duality.
On the other hand, we shall also give an example which shows that there are coactions which do not satisfy $E$\nb-duality
for any given $E\subseteq B(G)$ as above, thus giving a negative answer to \cite[Conjecture 6.12]{Kaliszewski-Landstad-Quigg:Exotic}.

\section{Some preliminaries on proper actions}

Recall that an action of a locally compact group $G$ on a locally compact (Hausdorff) space $X$
 is called \emph{proper} if the map
$$G\times X\to X\times X; (g,x)\mapsto (g x, x)$$
is proper in the sense that inverse images of compact sets are compact. Proper actions are extremely nice: their
orbit spaces $G\backslash X$ are always locally compact and Hausdorff and all stabilizers
$G_x=\{g\in G: g x=x\}$ are compact. Another property which characterizes properness is the so-called \emph{wandering condition}: for any two compact subsets $K_1, K_2\subseteq X$ the set
$C(K_1,K_2):=\{g\in G: gK_1\cap K_2\neq \emptyset\}$ is compact in $G$.
Let $\tau:G\to\Aut(\contz(X))$ denote the corresponding action of $G$ on the algebra $\contz(X)$ of continuous functions
on $X$ which vanish at $\infty$ given by
$$\big(\tau_g(f)\big)(x)=f(g^{-1}\cdot x).$$
Let $\contz(X)\rtimes_\tau G$ denote the crossed product for the action of $G$ on $\contz(X)$ (note that for proper actions
on spaces the full and reduced crossed products coincide, so we can take either construction here).
If the action of $G$ on $X$ is proper and \emph{free} (\ie, all stabilizers are trivial),
it follows from work of Green and Rieffel (\eg see
\cite{Green:algebras_transformation_groups, Rieffel:Applications_Morita}) that there is a canonical $\contz(G\backslash X) - \contz(X)\rtimes_\tau G$ Morita equivalence
constructed as follows:
Let $B_0=\contc(G, \contz(X))$ be viewed as a dense subalgebra of $\contz(X)\rtimes_\tau G$ and let
$E_0=\contc(G\backslash X)\subseteq \contz(G\backslash X)$. Then $\F_c(X)\defeq \contc(X)$ can be made into
an $E_0-B_0$ pre-imprimitivity bimodule by defining left and right $E_0$- and $B_0$-valued
inner products and left and right actions of $E_0$ and $B_0$ on $\F_c(X)$, respectively, given by
\begin{align*}
\bbraket{\xi}{\eta}_{B_0}(t,x)&=\Delta(t)^{-1/2}\overline{\xi(x)}\eta(t^{-1} x)\\
_{E_0}\bbraket{\xi}{\eta}(G x)&=\int_G\eta(t^{-1} x) \overline{\xi(t^{-1} x)}\,dt\\
(\xi\cdot \varphi)(x)&=\int_G\Delta(t)^{-1/2}\xi(t^{-1} x)\varphi(t^{-1}, t^{-1} x))\,dt,\quad\text{and}\\
f\cdot \xi(x)&=f(G x)\xi(x)
\end{align*}
for all $\xi,\eta\in \F_c(X), \varphi\in B_0$ and $f\in E_0$. The pre-imprimitivity bimodule $\F_c(X)$ completes to give
a $\contz(G\backslash X)-\contz(X)\rtimes_\tau G$ imprimitivity bimodule $\F(X)$.
If the action of $G$ is not free, we still get a Hilbert $\contz(G\backslash X)-\contz(X)\rtimes_\tau G$-bimodule $\F(X)$,
but the $\contz(X)\rtimes_{\tau}G$-valued inner product will not be full. But if we define
$$I_X:=\cspn\{\bbraket{\xi}{\eta}_{\contz(X)\rtimes G}: \xi,\eta\in \F(X)\},$$ then $I_X$ is a closed ideal in
$\contz(X)\rtimes G$ such that $\F(X)$ becomes a $\contz(G\backslash X)-I_X$ imprimitivity bimodule.

There have been many attempts to extend the notion of properness to \cstar{}dynamical systems
$(A,G,\alpha)$ and to obtain analogues of the above Morita equivalences.
The weakest notion for proper actions is due to Rieffel (see \cite{Rieffel:Proper}, \cite{Rieffel:Integrable_proper}),
and his concept has been studied in a number of papers by several authors (\eg see \cite{Rieffel:Proper, Meyer:Generalized_Fixed, anHuef-Raeburn-Williams:ProperActions, Kaliszewski-Quigg-Raeburn:ProperActionsDuality, anHuef-Raeburn-Williams:Symmetric, Huef-Kaliszewski-Raeburn-Williams:Naturality_Rieffel, Rieffel:Integrable_proper}).
In \cite{Rieffel:Proper} Rieffel says that an action $\alpha:G\to \Aut(A)$ is \emph{proper}, if there exists a dense $\alpha$-invariant
subalgebra $A_c$ (playing the role of $\contc(X)$) such that the following conditions are satisfied:
\begin{itemize}
\item For all $a,b\in A_c$ the functions $t\mapsto \Delta(t)^{-1/2} a^*\alpha_t(b)$ and $t\mapsto a^*\alpha_t(b)$ lie in $L^1(G,A)$, and
\item for any $a,b\in A_c$ there is a unique element $_{E_0}\bbraket{a}{b}$ in
$$\M(A_c)^{G,\alpha}=\{m\in \M(A)^{G,\alpha}: mA_c, A_cm\subseteq A_c\}$$ such that for all $c\in A_c$ we have
$$_{E_0}\bbraket{a}{b} c= \int_G \alpha_t(ab^*)c\, dt.$$
\end{itemize}
Notice that, different from \cite{Rieffel:Proper}, we let $_{E_0}\bbraket{a}{b}$ act on the \emph{left} of $A_c$. Thus our constructions are \emph{dual} to the ones performed by Rieffel. A similar approach which is consistent with our constructions (up to factors involving modular functions) has been used by Meyer in \cite{Meyer:Generalized_Fixed}.

Under the above conditions, Rieffel shows that $A_c$ equipped with the $A\rtimes_{\alpha,r}G$-valued inner product
$$\bbraket{a}{b}_{A\rtimes_rG}=(t\mapsto \Delta(t)^{-1/2} a^*\alpha_t(b))\in L^1(G,A)\subseteq A\rtimes_rG$$
completes to give a Hilbert $A\rtimes_{\alpha,r}G$-module $\F(A)_r$ such that its \cstar{}algebra of compact operators can be naturally identified
with the generalized fixed-point algebra $A^{G,\alpha}$ defined as
$$A^{G,\alpha}:=\cspn\{_{E_0}\bbraket{a}{b}: a,b\in A_c\}\subseteq \M(A).$$
A proper action in this sense is called \emph{saturated} if the $A\rtimes_{\alpha,r}G$-valued inner product on $\F(A)_r$
is full, so that $\F(A)_r$ becomes a $A^{G,\alpha}-A\rtimes_{\alpha,r}G$ imprimitivity bimodule.
In general, $\F(A)_r$ will be an imprimitivity bimodule between
$A^{G,\alpha}$ and the ideal $I(A)_r:=\cspn\{\bbraket{a}{b}_{A\rtimes_rG}: a,b\in A_c\}$.
The notion of \emph{integrability} introduced in \cite{Rieffel:Integrable_proper} even extends the above notion of properness, but lacks
a suitable definition of generalized fixed-point algebras and corresponding Morita equivalences. Another
extension of Rieffel's theory is given by the theory of continuously square-integrable actions due to Meyer (see \cite{Meyer:Generalized_Fixed}), which seems to be the weakest notion of properness allowing
a construction of a \emph{reduced} generalized fixed-point algebra. However it is not clear whether such proper actions allow the construction of \emph{full} generalized fixed-point algebras.

Assume now that $A$ is a weakly proper $X\rtimes G$-algebra. Recall that this means that $X$ is a proper $G$\nb-space and
that $A$ is endowed with an action $\alpha:G\to \Aut(A)$ and a $G$\nb-equivariant nondegenerate
\Star{}homomorphism $\phi:\contz(X)\to \M(A)$. We shall sometimes simply say that $(A,\alpha)$ is a weak $X\rtimes G$-algebra.
It has been shown by Rieffel in \cite[Propositon 5.7]{Rieffel:Integrable_proper}
that $A_c:=\phi(\contc(X))A\phi(\contc(X))$ provides a dense subalgebra as in the above discussion
and therefore the action of $G$ on $A$ is proper in Rieffel's sense.

To see that the notion of weakly proper $X\rtimes G$-algebras is quite general we should remark that every action $\alpha:G\to \Aut(A)$ is Morita equivalent
to a weakly proper $G\rtimes G$\nb-action (with $G$ acting on itself by right translation): simply consider the
action $\alpha\otimes \Ad\rho:G\to A\otimes \K(L^2G)$, where $\rho$ denotes the right regular representation,
and then $1_A\otimes M: \contz(G)\to \M(A\otimes \K(L^2G))$ provides the desired $G$\nb-equivariant homomorphism, where $M:\contz(G)\to \Lb(L^2G)$ denotes the
representation by multiplication operators. We should also note that weakly proper $X\rtimes G$-algebras have been studied
in several papers (using different terminology) in the context of Rieffel-proper actions (\eg see \cite{Huef-Kaliszewski-Raeburn-Williams:Naturality_Rieffel, Huef-Raeburn-Williams:FunctorialityGFPA, Kaliszewski-Quigg-Raeburn:ProperActionsDuality}).

In what follows, we often write $f\cdot a$ (resp. $a\cdot f)$ for the element
$\phi(f)a$ (resp. $a\phi(f)$) if $a\in \M(A)$ and $f\in \contb(X)$.

\begin{definition}\label{def-dense-algebra}
Suppose that $A$ is a weakly proper $X\rtimes G$-algebra. Let $A_c:=\contc(X)\cdot A\cdot \contc(X)$ and let
$$\widetilde{A}_c^{G,\alpha}:=\{ m\in \M(A)^{G,\alpha} : m\cdot f, f\cdot m\in A_c\;\mbox{ for all } f\in \contc(X)\}.$$
Then the \emph{generalized fixed-point algebra for $A$ with compact supports} is the algebra
$$A_c^{G,\alpha}:=\contc(G\backslash X)\cdot \widetilde{A}_c^{G,\alpha}\cdot \contc(G\backslash X).$$
\end{definition}

It is straightforward to check that $A_c^{G,\alpha}$ is a \Star{}subalgebra of $\M(A)$.
In the following proposition we let $\|\cdot\|_\pn$ denote any \emph{crossed-product norm}, meaning any given
\cstar{}norm on $\contc(G,A)$ such that $\|\cdot\|_r\leq \|\cdot\|_\pn\leq \|\cdot \|_u$, where $\|\cdot\|_r$ denotes the \emph{reduced
norm} given by the (left) regular representation $\Lambda_A$ of $\contc(G,A)$ on $L^2(G,A)$ and $\|\cdot \|_u$ denotes the universal
(or maximal) norm on $\contc(G, A)$. We then write $A\rtimes_{\alpha,\pn}G$ for the completion of $\contc(G,A)$
with respect to $\|\cdot\|_\pn$ and call $A\rtimes_{\alpha,\pn}G$ the \emph{$\pn$\nb-crossed product} of $(A,G,\alpha)$.
Observe that such \emph{exotic} crossed products correspond to quotients of $A\rtimes_{\alpha,\un}G$ by ideals contained in $\ker(\Lambda_A)$.
With this notation we get:

\begin{proposition}\label{prop-full}
Suppose that $A$ is a weakly proper $X\rtimes G$-algebra.
Let $\F_c(A)=\contc(X)\cdot A {\ }$\footnote{Note that $\F_c(A)$ is a vector space, since if
$a=\sum_{k=1}^l f_i\cdot a_i$ is a linear combination, then  $a=f\cdot a$ for any $f\in C_c(X)$
such that $f\equiv 1$ on $\cup\{\supp(f_i): 1\leq k\leq l\}$.} and let $B_0=\contc(G,A)$ be viewed as a dense subalgebra of $A\rtimes_{\alpha,\pn}G$.
Then there is a $B_0$-valued inner product on $\F_c(A)$ and a right action of $B_0$ on $\F_c(A)$ defined as
\begin{equation}\label{inner-prod}
\begin{split}
\bbraket{\xi}{\eta}_{B_0}(t)&=\Delta(t)^{-1/2} \xi^*\alpha_t(\eta)\quad\text{and}\\
\xi\cdot \varphi&=\int_G \Delta(t)^{-1/2}\alpha_t\big(\xi\varphi(t^{-1})\big)\,dt
\end{split}
\end{equation}
which turns $\F_c(A)$ into a pre-Hilbert $B_0$-module. These operations extend to the
completion $\F_\pn(A)$ with respect to the norm $\|\xi\|_\pn:=\|\bbraket{\xi}{\eta}_{B_0}\|_\pn^{1/2}$, so that
$\F_\pn(A)$ becomes a Hilbert $A\rtimes_{\alpha,\pn}G$-module.
Moreover, there is a faithful \Star{}homomorphism $\Psi_\pn:A_c^{G,\alpha}\to \K(\F_\pn(A))$ with dense image,
given by
\begin{equation}
\label{left-action}
\Psi_\pn(m) \xi=m\xi\quad\mbox{for all } m\in A_c^{G,\alpha}, \xi\in \F_c(A),
\end{equation}
where on the right hand side we use multiplication inside $\M(A)$.
\end{proposition}

Before we start with the proof, we want to give a definition of an $A_c^{G,\alpha}$-valued inner product on
$\F_c(A)$ which will play a crucial role in the proof of the proposition. Recall that $A_c:=\contc(X)\cdot A\cdot \contc(X)$.
The following lemma is basically due to Rieffel and Kaliszewski-Quigg-Raeburn (see the proof of \cite[Theorem 5.7]{Rieffel:Integrable_proper} and \cite[\S 2]{Kaliszewski-Quigg-Raeburn:ProperActionsDuality}). Recall first that for any proper $G$\nb-space $X$, there exists a surjective linear map
$$\EE_\tau: \contc(X)\to \contc(G\backslash X);\quad\EE_\tau(f)(G x)=\int_G f(t^{-1} x)\, dt.$$
This extends to weakly proper $X\rtimes G$-algebras $A$ as follows:

\begin{lemma}\label{lem-left-inner-product}
There is well-defined surjective linear map $\EE:A_c\to A_c^{G,\alpha}$ such that
\begin{equation}\label{eq-E}
\EE(a)c=\int_G \alpha_t(a)c\, dt \quad \mbox{ for all } c\in A_c.
\end{equation}
The map $\EE$ has the following properties:
\begin{enumerate}
\item We have $\EE(a^*)=\EE(a)^*$, $\EE(\psi\cdot a)=\psi\cdot \EE(a)$ and $\EE(a\cdot \psi)=\EE(a)\cdot \psi$
for all $\psi\in \contc(G\backslash X)$ and $a\in A_c$.
\item For all $m\in \tilde{A}_c^{G,\alpha}$, $a\in A_c$ and $f\in \contc(X)$ we have $\EE(ma)=m\EE(a)$, $\EE(m\cdot f)=m\cdot \EE_\tau(f)$,
$\EE(am)=\EE(a)m$ and $\EE(f\cdot m)=\EE_\tau(f)\cdot m$.
\item For every $m\in A_c^{G,\alpha}$ there exists $f\in \contc(X)$ such that $\EE(m\cdot f)=\EE(f\cdot m)= m$.
\item For all pairs $f,g\in \contc(X)$ the linear map $\EE_{f,g}:A\to \M(A); a\mapsto \EE(f\cdot a\cdot g)$ is norm continuous.
\end{enumerate}
Moreover, there is an $A^{G,\alpha}_c$-valued inner product on $\F_c(A)$ given by
$$_{A_c^{\alpha}}\bbraket{\xi}{\eta}:=\EE(\xi\eta^*)$$
which then satisfies the equation
\begin{equation}\label{eq-leftinner}
\xi\cdot\bbraket{\eta}{\zeta}_{B_0}=_{A_c^{\alpha}}\!\bbraket{\xi}{\eta}\cdot\zeta
\end{equation}
for all $\xi,\eta,\zeta\in \F_c(A)$, where the left and right actions and the $B_0$-valued inner
products are as in the proposition.
\end{lemma}
\begin{proof} The fact that formula (\ref{eq-E}) determines a unique element in $\M(A)^{G,\alpha}$ has been shown
in the proof of \cite[Proposition 5.7]{Rieffel:Integrable_proper} and it follows easily from this formula that
$\EE(a^*)=\EE(a)^*$ for all $a\in A_c$. The formulas in (1) and (2) are easy to check
but can also be found in \cite[Lemma 2.1 and Lemma 2.3]{Kaliszewski-Quigg-Raeburn:ProperActionsDuality}.

It is shown in \cite[Lemma 2.3]{Kaliszewski-Quigg-Raeburn:ProperActionsDuality} that $f\cdot \EE(a), \EE(a)\cdot f\in A_c$.
So $\EE(a)\in \tilde{A}_c^{G,\alpha}$. Now if $a=f\cdot b\cdot g$ for some $f,g\in \contc(X)$ and $b\in A$ and if
$ \psi\in \contc(G\backslash X)^+$ such that
$\psi\equiv 1$ on $G\cdot (\supp f\cup \supp g)$, then $a=\psi\cdot a\cdot \psi$ and (1) implies
$$\EE(a)=\EE(\psi\cdot a \cdot \psi)=\psi\cdot \EE(a)\cdot \psi$$
from which it follows that $\EE$ takes values in the algebra $A_c^{G,\alpha}$. On the other hand,
if $m\in A_c^{G,\alpha}$, we can write $m=\psi\cdot m\cdot \psi$ for some function $\psi\in \contc(G\backslash X)$.
Now, choose a function $f\in \contc(X)$ such that $\EE_\tau(f)=\psi$. Then it follows from (2)
that $\EE(m\cdot f)=m\cdot \EE_\tau(f)=m$ and similarly that $\EE(f\cdot m)=m$. Since $m\cdot f\in A_c$, this shows
in particular that $\EE(A_c)=A_c^{G,\alpha}$.

After having observed surjectivity of $\EE:A_c\to A_c^{G,\alpha}$ the properties for the
$A_c^{G,\alpha}$-valued inner product follow from the proof of \cite[Proposition 5.7]{Rieffel:Integrable_proper}.

Finally, (4) follows from \cite[Lemma 3.6]{Quigg:Landstad_duality}.
\end{proof}

The map $\EE:A_c\to A_c^{G,\alpha}$ will be sometimes written as an integral $\EE(a)=\int_G^\st\alpha_t(a)dt$.
It can be shown that $\EE(a)$ is, indeed, an integral in the sense of Pettis with respect to the strict topology in $\M(A)$ (\cf~\cite[Proposition~2.2]{Buss:generalized_Fourier}).

\begin{proof}[Proof of Proposition \ref{prop-full}]
Recall from \cite[Lemma 2.16]{Raeburn-Williams:Morita_equivalence} that $\F_c(A)$ is a pre-Hilbert $B_0$-module if the right action of $B_0$ on $\F_c(A)$
and the $B_0$-valued inner product satisfy the following conditions:
\begin{itemize}
\item $\bbraket{\cdot}{\cdot}_{B_0}$ is $B_0$-linear in the
second variable,
\item $\bbraket{\xi}{\eta}_{B_0}^*=\bbraket{\eta}{\xi}_{B_0}$, and
\item $\bbraket{\xi}{\xi}_{B_0}\geq 0$ as an element
of $A\rtimes_{\alpha,\pn}G$ for all $\xi,\eta\in \F_c(A)$.
\end{itemize}
From these conditions the first two are well known (and easy to check), so
we concentrate on the positivity condition. Since $A\rtimes_{\alpha,\pn}G$ is a quotient of the universal crossed
product $A\rtimes_\alpha G=A\rtimes_{\alpha,u}G$, it suffices to show positivity in this case.

For this let $\xi\in\F_c(A)$ be given. Then we may write $\xi=f\cdot a$ for some $f\in \contc(X)$ and $a\in A$. Recall
that the canonical embedding $\iota_A:A\to \M(A\rtimes_\alpha G)$ is determined by the equations
$$(\iota_A(a)\varphi)(t)=a\varphi(t)\quad\text{and} \quad (\varphi \iota_A(a))(t)=\varphi(t)\alpha_t(a) \quad\text{for $\varphi\in B_0$}.$$
On the other hand, the $G$\nb-equivariant \Star{}homomorphism $\phi:\contz(X)\to \M(A)$ induces a \Star{}homomorphism
$\phi\rtimes G: \contz(X)\rtimes_\tau G\to \M(A\rtimes_{\alpha}G)$ such that, for a function
$g\in \contc(G, \contz(X))$ and $h\in \contc(G,A)$, we have
$\phi\rtimes G(g)h=g*h$, where $g*h(t)\defeq \int_G g(s)\cdot\alpha_s(h(s^{-1}t))ds$. Using this one easily checks
that
$$\big(\iota_A(a)(\phi\rtimes G)(g)\big)(t)=a\cdot g(t)\quad\text{and}\quad \big((\phi\rtimes G)(g) \iota_A(a)\big)(t)=g(t)\cdot \alpha_t(a)$$
for all $g\in \contc(G, \contz(X))$ and $a\in A$. We then compute
\begin{align*}
\bbraket{f\cdot a}{f\cdot a}_{B_0}(t)&=\Delta(t)^{-1/2} (f\cdot a)^*\alpha_t(f\cdot a)
= a^*\cdot\big(\Delta(t)^{-1/2} \cdot (\bar{f}\tau_t(f))\big)\cdot\alpha_t(a)\\
&=\big(i_a(a^*) (\phi\rtimes G)\big(\bbraket{f}{f}_{\contz(X)\rtimes G}\big)\iota_A(a)\big)(t).
\end{align*}
Hence, we get
$$\bbraket{\xi}{\xi}_{B_0}=\bbraket{f\cdot a}{f\cdot a}_{B_0}=\iota_A(a^*)(\phi\rtimes G)\big(\bbraket{f}{f}_{\contz(X)\rtimes G}\big) \iota_A(a)$$
inside $\M(A\rtimes_{\alpha}G)$. Since the $\contz(X)\rtimes G$-valued inner product on $\F(X)=\overline{\contc(X)}$ is known
to be positive, it follows that $\bbraket{f\cdot a}{f\cdot a}_{B_0}$ is positive in $A\rtimes_{\alpha}G$.

At this point we should also note that $\bbraket{\xi}{\xi}_{B_0}\neq 0$ if $\xi\neq 0$, which follows from the fact
that $\bbraket{\xi}{\xi}_{B_0}(e)=\xi^*\xi>0$ in $A$ and that the imbedding of $\contc(G,A)$ into $A\rtimes_{\alpha,\pn}G$
is injective (since this is already true for the reduced crossed product).

It now follows that $\F_c(A)$ completes to give a Hilbert $A\rtimes_{\alpha,\pn}G$-module $\F_\pn(A)$. The fact that
$A_c^{G,\alpha}$ acts as a dense subalgebra of $\K(\F_\pn(A))$ follows from Lemma \ref{lem-left-inner-product}:
Equation (\ref{eq-leftinner}) implies that for any pair $\xi,\eta\in \F_c(A)$ the left inner product
$_{A_c^{\alpha}}\!\bbraket{\xi}{\eta}=\EE(\xi\eta^*)$ acts as the compact operator $\Theta_{\xi,\eta}$ on $\F_\pn(A)$.
In particular, this operator is non-zero if $\xi,\eta\neq 0$. Since
$\F_c(A)\F_c(A)^*=A_c$ and since by Lemma \ref{lem-left-inner-product} we have $\EE(A_c)=A_c^{G,\alpha}$
we see that left multiplication of $A_c^{G,\alpha}$ on $\F_c(A)$ induces an injective \Star{}homomorphism
of $A_c^{G,\alpha}$ onto a dense subalgebra of $\K(\F_\pn(A))$.
\end{proof}

\begin{definition}\label{def-max-fix}
Suppose that $A$ is a weakly proper $X\rtimes G$-algebra and let $\|\cdot\|_\pn$ be any crossed-product norm on $\contc(G,A)$.
Then the $\pn$-fixed-point algebra of $A$ is defined as
$$A_\pn^{G,\alpha}:=\K(\F_\pn(A)).$$
In other words, $A_\pn^{G,\alpha}$ is the completion of $A_c^{G,\alpha}$ with respect to the norm
induced by the left action of $A_c^{G,\alpha}$ on $\F_\pn(A)$.
\end{definition}

We now want to give an alternative construction of the module $\F_\pn(A)$. For this recall
that $\F(X)=\F(\contz(X))$ denotes the canonical Hilbert $\contz(G\backslash X)-\contz(X)\rtimes_\tau G$-bimodule
for the proper $G$\nb-space $X$. To make notations a bit more convenient, we shall
 write from now on $\bbraket{f}{g}_X$ instead of $\bbraket{f}{g}_{C_0(X)\rtimes_\tau G}$ for
the $C_0(X)\rtimes_\tau G$-valued inner product on $\F(X)$.

\begin{proposition}\label{prop-factor}
Suppose that $A$ is a weakly proper $X\rtimes G$-algebra. Let
$\phi\rtimes_\pn G:\contz(X)\rtimes_{\tau}G\to \M(A\rtimes_{\alpha,\pn}G)$ denote the
canonical nondegenerate \Star{}homomorphism induced from the $G$\nb-equivariant
\Star{}homomorphism $\phi:\contz(X)\to \M(A)$. Then there is an isomorphism
of Hilbert $A\rtimes_{\alpha,\pn}G$-modules
$$\Psi_\pn: \F(X)\otimes_{\contz(X)\rtimes G} (A\rtimes_{\alpha,\pn}G)\to \F_\pn(A)$$
given on elementary tensors $f\otimes \varphi\in \contc(X)\otimes \contc(G,A)$ by
\begin{equation}\label{eq-decom}
\Psi_\pn(f\otimes\varphi)=\int_G\Delta(t)^{-1/2} \alpha_t(f\cdot \varphi(t^{-1}))\,dt.
\end{equation}
In particular, it follows that the $A\rtimes_{\alpha,\pn}G$-valued inner product on $\F_\pn(A)$
is full whenever $G$ acts freely on $X$.
\end{proposition}
\begin{proof}
An easy computation shows that $\Psi_\pn(f\otimes \varphi)\cdot\psi=\Psi_\pn\big(f\otimes (\varphi*\psi)\big)$ for all
$\varphi,\psi\in \contc(G,A)$ and $f\in \contc(X)$, so that the result will follow if we can show that
$\Psi_\pn$ is an isometry with dense image in $\F_\pn(A)$. To see that it is isometric it suffices
to check $\bbraket{\Psi_\pn(f\otimes \varphi)}{\Psi_\pn(g\otimes\psi)}_{B_0}=\bbraket{f\otimes\varphi}{g\otimes\psi}_{B_0}$
for all $f\otimes \varphi, g\otimes\psi\in \contc(X)\otimes \contc(G,A)$. For this we compute
\begin{align*}
\bbraket{f\otimes\varphi&}{g\otimes\psi}_{B_0}(t)=\bbraket{\varphi}{\bbraket{f}{g}_{X}\cdot \psi}_{B_0}(t)\\
&=\varphi^* * \big(\bbraket{f}{g}_X*\psi)(t)\\
&=\int_G\int_G\Delta(s^{-1})\alpha_s\big(\varphi(s^{-1})^* \Delta(r)^{-1/2} \bar{f}\tau_r(g)\alpha_r(\psi(r^{-1}s^{-1}t))\big)\,dr\,ds\\
&\stackrel{r\mapsto s^{-1}tr}{=}\int_G\int_G\Delta(str)^{-1/2} \alpha_s\big(\varphi(s^{-1})^* \bar{f}\tau_{s^{-1}tr}(g)\alpha_{s^{-1}tr}(\psi(r^{-1}))\big)\,dr\,ds\\
&=\int_G\int_G\Delta(str)^{-1/2} \alpha_s\big(\varphi(s^{-1}))^* \tau_s(\bar{f})\tau_{tr}(g)\alpha_{tr}(\psi(r^{-1}))\,dr\,ds\\
&=\Delta(t)^{-1/2}\Psi_\pn(f\otimes\varphi)^*\alpha_t\big(\Psi_\pn(g\otimes \psi)\big)\\
&=\bbraket{\Psi_\pn(f\otimes \varphi)}{\Psi_\pn(g\otimes\psi)}_{B_0}(t).
\end{align*}
Thus the proposition will follow if we can show that $\Psi_\pn$ has dense image in $\F_\pn(A)$. We postpone the
proof for this to Lemma \ref{lem-factor} below.
\end{proof}

In what follows we need to argue at several instances with certain inductive limit topologies on
$B_0=\contc(G,A)$, $A_c=\contc(X)\cdot A\cdot \contc(X)$, $\F_c(A)=\contc(X)\cdot A$ and $A_c^{G,\alpha}$.
In case of $B_0=\contc(G,A)$, this is the usual definition of uniform convergence of nets $(\varphi_i)$ with
supports $\supp\varphi_i$ lying in a fixed compact set for all $i\in I$. In the other cases the definitions are
as follows:

\begin{definition}\label{def-ind-limit}
{\bf (1)} If $(a_i)_{i\in I}$ is a net in $\F_c(A)$ (resp. in $A_c$),
we say that $a_i\to a\in \F_c(A)$ (resp. in $A_c$) in the \emph{inductive limit topology}, if $a_i\to a$ in $A$ in the norm topology
and
there exists some $f\in \contc(X)$ such that $a_i=f\cdot a_i$ (resp. $a_i=f\cdot a_i\cdot f$) for all $i\in I$.

{\bf (2)} If $(m_i)_{i\in I}$ is a net in $A_c^{G,\alpha}$, we say that $m_i\to m\in A_c^{G,\alpha}$ in the \emph{inductive limit topology}
if $m_i\to m$ in $\M(A)$ in norm and the following conditions hold:
\begin{enumerate}[(i)]
\item There exists $\psi\in \contc(G\backslash X)$ such that $m_i=\psi\cdot m_i\cdot \psi$ for all $i\in I$, and
\item $m_i\cdot f\to m\cdot f$ and $f\cdot m_i\to f\cdot m$ in the inductive limit topology of $A_c$ for all $f\in \contc(X)$.
\end{enumerate}
\end{definition}

\begin{lemma}\label{lim-inductive-lim-top}
Let $A$ be a weakly proper $X\rtimes G$-algebra.
Then all pairings in the pre-Hilbert $A_c^{G,\alpha}-B_0$-bimodule $\F_c(A)$ are jointly continuous with respect to
the respective inductive limit topologies.
\end{lemma}
\begin{proof}
There are four pairings to consider: the $A_c^{G,\alpha}$- and $B_0$-valued inner products on $\F_c(A)$ and
the pairings coming from the left and right actions of $A_c^{G,\alpha}$ and $B_0$ on $\F_c(A)$, respectively.
We do the inner products here and leave the action-pairings to the reader.

For the $B_0$-valued inner product on $\F_c(A)$, let $\xi_i\to \xi$ and $\eta_i\to\eta$ in the inductive limit topology of $\F_c(A)$
and let $f\in \contc(X)$
such that $f\cdot \xi_i=\xi_i$ and $f\cdot \eta_i =\eta_i$ for all $i\in I$.
Then the computations in the proof of Proposition \ref{prop-full} show that
$$\bbraket{\xi_i}{\eta_i}_{B_0}(t)=\bbraket{f\cdot \xi_i}{f\cdot\eta_i}_{B_0}(t)
=\xi_i\cdot \bbraket{f}{f}_X(t)\cdot \alpha_t(\eta_i)$$
which clearly converges to $t\mapsto \xi\cdot \bbraket{f}{f}_X(t)\cdot \alpha_t(\eta)=\bbraket{\xi}{\eta}_{B_0}(t)$
in the inductive limit topology of $B_0$.

For the left inner product $_{A_c^\alpha}\bbraket{\cdot}{\cdot}$ we first observe that if $\xi_i\to \xi$ and $\eta_i\to \eta$
in the inductive limit topology of $\F_c(A)$, then $\xi_i\eta_i^*\to \xi\eta^*$ in the inductive limit topology of $A_c$
and it follows then from Lemma \ref{lem-left-inner-product} (4) that $_{A_c^\alpha}\bbraket{\xi_i}{\eta_i}=\EE(\xi_i\eta_i^*)
\to \EE(\xi\eta^*)={_{A_c^\alpha}\bbraket{\xi}{\eta}}$ in norm. Moreover, if $f\in \contc(X)$ such that
$f\cdot (\xi_i\eta_i^*)\cdot f=\xi_i\eta_i^*$ for all $i\in I$ and if $\psi\in \contc(G\backslash X)$ such that
$\psi\equiv 1$ on $G\cdot \supp f$, then it follows that $\psi\cdot{_{A_c^\alpha}\bbraket{\xi_i}{\eta_i}}\cdot\psi
={_{A_c^\alpha}\bbraket{\xi_i}{\eta_i}}$ for all $i\in I$. Finally, for fixed $g\in \contc(X)$, we have
\begin{align*}
g\cdot {_{A_c^\alpha}\bbraket{\xi_i}{\eta_i}}&
= g\cdot \EE(f\cdot \xi_i\eta_i^*\cdot f)\\
&=\int_G^\st g\cdot \alpha_t(f\cdot \xi_i\eta_i^*\cdot f)\,dt\\
&=\int_G^\st g\cdot \tau_t(f)\cdot \alpha_t(\xi_i\eta_i^*)\cdot \tau_t(f)\,dt.
\end{align*}
If we choose a function $h\in \contc(X)$ such that $h\equiv 1$ on $K\cdot \supp f\cup \supp g$, where
$K=\supp(t\mapsto g\tau_t(f)))\subseteq G$, we see that
$h\cdot(g\cdot {_{A_c^\alpha}\bbraket{\xi_i}{\eta_i}})\cdot h= g\cdot {_{A_c^\alpha}\bbraket{\xi_i}{\eta_i}}$ for all $i\in I$
and it is easy to check from the above formulas that $g\cdot {_{A_c^\alpha}\bbraket{\xi_i}{\eta_i}}\to g\cdot {_{A_c^\alpha}\bbraket{\xi}{\eta}}$ in norm. Hence $g\cdot {_{A_c^\alpha}\bbraket{\xi_i}{\eta_i}}\to g\cdot {_{A_c^\alpha}\bbraket{\xi}{\eta}}$
in the inductive limit topology and a similar argument shows that
$ {_{A_c^\alpha}\bbraket{\xi_i}{\eta_i}}\cdot g\to {_{A_c^\alpha}\bbraket{\xi_i}{\eta_i}}\cdot g$ in the inductive limit topology
of $A_c$.
\end{proof}

\begin{lemma}\label{lem-factor}
Let $\xi=f\cdot a\in \F_c(A)$ be fixed. Then, there exists a net $\varphi_i\in \contc(G,A)$ such that
$\Psi_\pn(f\otimes \varphi_i)=\int_G \Delta(t)^{-1/2}\alpha_t(f\cdot\varphi_i(t^{-1}))\,dt$ converges to $\xi$ in the inductive limit topology of
$\F_c(A)$. In particular, the map $\Psi_\pn:\F(X)\otimes_{\contz(X)\rtimes G}(A\rtimes_{\alpha,\pn}G)\to \F_\pn(A)$ of
Proposition \ref{prop-factor} is surjective.
\end{lemma}
\begin{proof} Let $\mathcal U$ be a neighborhood base of the identity $e\in G$ consisting of symmetric compact neighborhoods and
let $\tilde\varphi_U\in \contc(G)^+$ with $\supp \tilde\varphi_U\subseteq U$, $\tilde\varphi(t)=\tilde\varphi(t^{-1})$ for all $t\in G$
and $\int_G\tilde\varphi_U(t)\,dt=1$.
Let $\varphi_U\in \contc(G,A)$ be defined as $\varphi_U(t)=\Delta(t)^{1/2}\alpha_{t^{-1}}(a)\tilde\varphi(t)$.
It is then easy to check that $(\varphi_U)_{U\in \mathcal U}$ does the job.

Since, by Lemma \ref{lim-inductive-lim-top}, the inductive limit topology is stronger than the
norm topology on $\F_c(A)\subseteq \F_\pn(A)$, it follows that $\Psi_\pn$ has dense image in $\F_\pn(A)$.
Since it is isometric, it is surjective.
\end{proof}

\section{Representations of generalized fixed-point algebras}\label{sec-reps}

Assume that $A$ is a weakly proper $X\rtimes G$-algebra. In this section we describe the representations of
generalized fixed-point algebras $A_\pn^{G,\alpha}$ -- where $\|\cdot\|_\pn$ is some fixed crossed-product norm -- via the
induction process given by the partial $A_\pn^{G,\alpha}-A\rtimes_{\alpha,\pn}G$ imprimitivity bimodule $\F_\pn(A)$.

Recall that the (nondegenerate) representations of the maximal crossed product $A\rtimes_{\alpha}G$ are
the integrated forms $\pi\rtimes U$ of covariant representations $(\pi,U)$ of the system $(A,G,\alpha)$.
Throughout, we assume all representations to be nondegenerate and also allow representations on Hilbert modules (not only on Hilbert spaces).

Let $\|\cdot\|_\pn$ be a crossed-product norm on $\contc(G,A)$ and let $J_\pn$ be the kernel of the surjection $q_\pn\colon A\rtimes_\un G\to A\rtimes_\pn G$.
Then, representations of $A\rtimes_\pn G=(A\rtimes_\un G)/J_\pn$ correspond to representations of $A\rtimes_\un G$ vanishing on $J_\pn$,
\ie, covariant representations $(\pi,U)$ with $J_\pn\sbe \ker(\pi\rtimes U)$. Such a representation $(\pi,U)$ will be called a \emph{$\pn$-covariant representation}, and the corresponding representation of $A\rtimes_\pn G$ will be denoted by $\pi\rtimes_\pn U$.

Since $\F_\pn(A)$ is an imprimitivity $A^{G,\alpha}_\pn-\I_{A,\pn}$-bimodule, where $\I_{A,\pn}$ is the ideal $\cspn\{\bbraket{\xi}{\eta}_{A\rtimes_\pn G}: \xi,\eta\in \F_\pn(A)\}$ of $A\rtimes_{\pn}G$, the representations of $A_\pn^{G,\alpha}$ are (up to unitary equivalence) the representations which are induced via $\F_\pn(A)$ from the nondegenerate representations of $\I_{A,\pn}$. Since every nondegenerate representation of an ideal uniquely extends to a representation of the full algebra,
we may describe the representations of $A_\pn^{G,\alpha}$ as those representations which are
induced from representations of $A\rtimes_{\pn}G$ which restrict to nondegenerate representations
of $\I_{A,\pn}$.

For a given $\pn$-covariant representation $(\pi,U)$ of $(A,G,\alpha)$ on a Hilbert $B$-module $\Y$,
recall that the representation of $A^{G,\alpha}_\pn$ induced from $\pi\rtimes_\pn U$ via $\F_\pn(A)$ acts on the
Hilbert $B$-module $\Y_{\Ind (\pi\rtimes U)}:=\F_\pn(A)\otimes_{A\rtimes_{\pn}G}\Y$ via the
action of $A^{G,\alpha}_\pn$ on $\F_\pn(A)$. The following result gives an explicit description of this representation.

\begin{proposition}\label{prop-induced}
Let $(A,\alpha,\phi)$ be a weakly proper $X\rtimes G$-algebra, let $\|\cdot\|_\pn$ be a crossed-product norm on $\contc(G,A)$, and let $(\pi,U)$ be a $\pn$-covariant representation
of $(A,G,\alpha)$ on a Hilbert $B$-module $\Y$. Let $\H_0:=\tilde\pi(\contc(X))\Y$, where $\tilde\pi\defeq \pi\circ\phi$, equipped with the \textup(semi-positive definite\textup) $B$-valued inner product
\begin{equation}\label{eq-integral}
\braket{\xi}{\eta}_0:= \int_G \Delta(t)^{-1/2}\braket{\xi}{U_t\eta}_{B}\,dt
\end{equation}
and consider the left action of $A_c^{G,\alpha}$ on $\H_0$ given by
\begin{equation}\label{eq-left-action}
m\cdot \xi=\pi(m)\xi,
\end{equation}
where we extend $\pi$ to $\M(A)$ and use the fact that $A_c^{G,\alpha}\subseteq \M(A)$.
Then this action extends to a representation $\Ind^{G,\alpha}_\pn(\pi\rtimes U)$ of $A_\pn^{G,\alpha}$ on the Hausdorff completion
$\H=\Ind^{G,\alpha}_\pn(\Y)$ of $\H_0$ which is unitarily equivalent to the induced representation $\Ind^{\F_\pn(A)}(\pi\rtimes U)$
of $A_\pn^{G,\alpha}$ on $\Y_{\Ind(\pi\rtimes_\pn U)}$ via $\F_\pn(A)$.

In particular, if $\pi\rtimes_\pn U$ is faithful \textup(\ie, if $\ker(\pi\rtimes U)=J_\pn$\textup),
then $\Ind^{G,\alpha}_\pn$ is a faithful representation of $A^{G,\alpha}_\pn$.
\end{proposition}

\begin{proof}
First note that the integral in (\ref{eq-integral}) exists: since $\xi,\eta\in \pi(\contc(X))\Y$, we may choose
$f,g\in \contc(X)$ such that $\xi=\tilde\pi(f)\xi$ and $\eta=\tilde\pi(g)\eta$. Then
\begin{align*}
\int_G \Delta(t)^{-1/2}\braket{\xi}{U_t\eta}_{B}\,dt&=
\int_G \Delta(t)^{-1/2}\braket{\tilde\pi(f)\xi}{U_t\tilde\pi(g)\eta}_{B}\,dt\\
&=\int_G \Delta(t)^{-1/2}\braket{\xi}{\tilde\pi(\bar{f}\tau_t(g))U_t\eta}_{B}\,dt,
\end{align*}
hence we integrate a continuous function with compact support. We should also remark
that $m\cdot \xi\in \H_0$ for all $m\in A_c^{G,\alpha}$ and $\xi\in \H_0$. For this let $f\in \contc(X)$ such that
$\xi=\tilde\pi(f)\xi$ and choose $g\in \contc(X)$ such that $g\cdot m\cdot f=m\cdot f$ (such $g$ exists since
$m\cdot f\in A_c$). Then $m\cdot \xi=\tilde\pi(g)\pi(m\cdot f)\xi\in \H_0$.

In a next step, we use the decomposition $\F_\pn(A)=\F(X)\otimes_{\contz(X)\rtimes G}(A\rtimes_{\pn}G)$ of Proposition \ref{prop-factor}
to see that
\begin{equation}\label{eq-iso}
\begin{split}
\Y_{\Ind (\pi\rtimes_\pn U)}&=\F_\pn(A)\otimes_{A\rtimes_{\pn}G}\Y\\
&\stackrel{\Phi}{\cong}\big(\F(X)\otimes_{\contz(X)\rtimes G}(A\rtimes_{\pn}G)\big)\otimes_{A\rtimes_{\pn}G}\Y\\
&= \F(X)\otimes_{\contz(X)\rtimes G}\big((A\rtimes_{\pn}G)\otimes_{A\rtimes_{\pn}G}\Y\big)\\
&\stackrel{\Theta}{\cong} \F(X)\otimes_{\contz(X)\rtimes G}\Y
\end{split}
\end{equation}
Note that we denoted the first isomorphism in the above computation by $\Phi$ and the last one by $\Theta$.
For bookkeeping, the inverse
$$\Phi^{-1}:\F(X)\otimes_{\contz(X)\rtimes G}(A\rtimes_{\pn}G)\otimes_{A\rtimes_{\pn}G}\Y
\to \F_\pn(A)\otimes_{A\rtimes_{\pn}G}\Y$$
of the isomorphism $\Phi$ sends a triple elementary tensor
$$f\otimes \varphi\otimes \xi\in \contc(X)\otimes \contc(G,A)\otimes \Y$$
to the elementary tensor
$$\Psi_\pn(f\otimes \varphi)\otimes \xi\in \F_c(A)\otimes \Y$$
 with $\Psi_\pn(f\otimes \varphi)=\int_G\Delta(t)^{-1/2}\alpha_t(f\cdot \varphi(t^{-1}))\,dt$, and we have
$$\Theta(f\otimes \varphi\otimes \xi)=f\otimes \big(\pi\rtimes_\pn U(\varphi)\xi\big) \in \contc(X)\otimes \Y.$$

To proceed, we now consider the surjective linear map $V: \F_c(X)\odot \Y{\to} \H_0$
given on elementary tensors by $V(f\otimes \xi)=\tilde\pi(f)\xi$.
A short computation yields
\begin{align*}
\braket{f\otimes\xi}{g\otimes\eta}&=\braket{ \xi}{ \pi\rtimes_\pn U(\bbraket{f}{g}_X)\eta}\\
&= \Braket{\xi}{\int_G \Delta(t)^{-1/2}\tilde\pi(\bar{f}\tau_t(g))U_t\eta\, dt}\\
&=\int_G \Delta(t)^{-1/2}\braket{\tilde\pi(f)\xi}{U_t\tilde\pi(g)\eta}\,dt\\
&=\braket{\tilde\pi(f)\xi}{\tilde\pi(g)\eta}_0=\braket{V(f\otimes\xi)}{V(g\otimes\eta)}_0.
\end{align*}
It follows that $\braket{\cdot}{\cdot}_0$ is a positive semi-definite $B$-inner product
on $\H_0$ such that $V$ extends to a Hilbert $B$-module isomorphism between
$\F(X)\otimes_{\contz(X)\rtimes G}\Y$ and the Hausdorff completion of $\H_0$
with respect to this $B$-inner product. Thus, putting things together, we obtain an isomorphism of Hilbert $B$-modules
$$V\circ \Theta\circ \Phi: \Y_{\Ind(\pi\rtimes_\pn U)}\to \H.$$

We now have to check that this isomorphism intertwines the induced representation
$\Ind^{\F(A)}(\pi\rtimes_\pn U)$ with the representation determined by the left action of $A_c^{G,\alpha}$ on $\H_0$ as
in~\eqref{eq-left-action}. For this let $m\in A_c^{G,\alpha}$ and let
$ f\otimes \varphi\otimes \xi$ be a triple tensor in
$ \contc(X)\otimes \contc(G,A)\otimes \Y\subseteq
\F(X)\otimes_{\contz(X)\rtimes G}(A\rtimes_{\pn}G)\otimes_{A\rtimes_{\pn}G}\Y$.
We then get
$$\Ind^{\F(A)}(\pi\rtimes_\pn U)(m)\big(\Phi^{-1}(f\otimes\varphi\otimes \xi)\big)=m\cdot \Psi_\pn(f\otimes\varphi)\otimes \xi.$$
Choose a function $g\in \contc(X)$ such that $g\cdot m\cdot f=m\cdot f$, which exists since $m\cdot f\in A_c$.
Using the $G$\nb-invariance of $m$, we get
\begin{align*}
m\cdot \Psi_\pn(f\otimes\varphi)\otimes\xi&=\int_G\Delta(t)^{-1/2} \alpha_t(m\cdot f\cdot \varphi(t^{-1}))\,dt\otimes \xi\\
&=\int_G\Delta(t)^{-1/2} \alpha_t(g\cdot m\cdot f\cdot \varphi(t^{-1}))\,dt\otimes \xi\\
&=\Psi_\pn\big(g\otimes \iota_A^\pn(m\cdot f)\varphi\big)\otimes \xi\\
&=\Phi^{-1}\big(g\otimes \iota_A^\pn(m\cdot f)\varphi\otimes \xi\big).
\end{align*}
where $\iota_A^\pn:A\to \M(A\rtimes_{\pn}G)$ denotes the canonical inclusion.
The isomorphism $V\circ \Theta\circ \Phi: \Y_{\Ind(\pi\rtimes_\pn U)}\to \H$ sends this element
to
\begin{align*}
V\circ \Theta(g\otimes \iota_A^\pn(m\cdot f)\varphi\otimes \xi)&=
V(g\otimes\pi\rtimes_\pn U(\iota_A^\pn(m\cdot f)\varphi)\xi)\\
&=V(g\otimes \pi(m\cdot f)\pi\rtimes_\pn U(\varphi)\xi)\\
&=\pi(g\cdot m\cdot f)\pi\rtimes_\pn U(\varphi)\xi\\
&=\pi(m\cdot f)\pi\rtimes_\pn U(\varphi)\xi\\
&=\pi(m)\big(\tilde\pi(f)\pi\rtimes_\pn U(\varphi)\xi\big)\\
&=m\cdot \left(V\circ \Theta(f\otimes \varphi\otimes \xi)\right).
\end{align*}
This completes the proof of the first assertion in the proposition. The final assertion follows from the fact that the induction process via imprimitivity bimodules preserves faithfulness.
\end{proof}

Next, we are going to describe what happens if we induce to generalized fixed-point algebras
the covariant regular representation $\Lambda_A$ of $(A,\alpha)$ on the Hilbert $A$-module $L^2(G,A)$. Recall that $\Lambda_A$ is the integrated form of the covariant representation $(\tilde\alpha,1\otimes\lambda)$ on $L^2(G,A)\cong A\otimes L^2G$, where $\tilde\alpha(a)\xi|_t\defeq \alpha_{t^{-1}}(a)\xi(t)$ for all $a\in A$, $t\in G$ and $\xi\in \contc(G,A)\sbe L^2(G,A)$, and $\lambda\colon G\to \U(L^2G)$ is the (left) regular representation of $G$. Recall that $\ker(\Lambda_A)=J_\red\supseteq J_\pn$ so that $(\tilde\alpha,1\otimes\lambda)$ is a $\pn$-representation for every crossed-product norm $\|\cdot\|_\pn$ on $\contc(G,A)$. Hence $\Lambda_A$ factors through a representation $\Lambda_A^\pn$ of $A\rtimes_\pn G$ on $L^2(G,A)$ for every $\pn$ and $\Lambda_A^\pn$ is faithful \iff $\pn=\red$. Of course, if $\Lambda_A$ is faithful, then $\Lambda_A^\pn$ is faithful for every $\pn$. The following result describes the induced representation $\Ind_\pn^{G,\alpha}(\Lambda_A)$ of $A^{G,\alpha}_\pn$.

\begin{proposition}\label{prop:RegRepInduced}
With notations as above, there is an isomorphism
$$\Psi\colon \Ind^{G,\alpha}_\pn\big(L^2(G,A)\big)\congto A$$
of Hilbert $A$-modules which sends a function $\zeta\in\contc(X)\cdot\contc(G,A)\sbe \Ind^{G,\alpha}_\pn\big(L^2(G,A)\big)$ to $\Psi(\zeta)\defeq \int_G\Delta(t)^{-1/2}\zeta(t)dt\in A$.

Via this isomorphism, the representation $\Ind_\pn^{G,\alpha}(\Lambda_A)$ of $A^{G,\alpha}_\pn$ into $\M(A)=\Lb_A(A)$ is given by the extension  to $A_\pn^{G,\alpha}$ of the canonical inclusion map $A^{G,\alpha}_c\into \M(A)$. In particular, $A^{G,\alpha}_\red$ may be identified with the closure of $A_c^{G,\alpha}$ in $\M(A)$.
\end{proposition}
\begin{proof}
First notice that $\H_c\defeq \contc(X)\cdot\contc(G,A)$ is the (dense) subspace
$$\{f\cdot \xi\defeq \tilde\alpha(\phi(f))\cdot \xi: f\in \contc(X),\xi\in \contc(G,A)\}$$
of $\H=\Ind^{G,\alpha}_\pn\big(L^2(G,A)\big)$. For $f\in \contc(X)$ and $\xi\in \contc(G,A)$, the function $\zeta=f\cdot \xi$ belongs to $\contc(G,A)$ and is given by $\zeta(t)=f\cdot\xi(t)=\alpha_{t^{-1}}(\phi(f))\xi(t)$, so the integral $\psi(\zeta)=\int_G\Delta(t)^{-1/2}\zeta(t)dt$ makes sense and yields an element in $A$.
To prove the first assertion of the proposition, it is enough to check that $\Psi$ preserves $A$-inner products and has dense range.
Now, if $\eta,\zeta\in \H_c$, then
\begin{align*}
\braket{\eta}{\zeta}_0&=\int_G\Delta(t)^{-1/2}\braket{\eta}{(1\otimes\lambda_t)\zeta}_A dt\\
 &= \int_G\int_G\Delta(t)^{-1/2}\eta(s)^*\zeta(t^{-1}s) dt ds \\
 &\stackrel{t\mapsto st^{-1}}{=}\int_G\int_G\Delta(st)^{-1/2}\eta(s)^*\zeta(t) dt ds \\
 &= \left(\int_G\Delta(s)^{-1/2}\eta(s)dt\right)^*\left( \int_G \Delta(t)^{-1/2} \zeta(t) ds \right)\\
 &=\braket{\Psi(\eta)}{\Psi(\zeta)}.
\end{align*}
Therefore, $\Psi$ preserves the $A$-valued inner products. To see that $\Psi$ has dense range, it is enough to approximate elements in $f\cdot a\in \contc(X)\cdot A$ by elements in $\im(\Psi)$ in the norm of $A$. For this let $\epsilon>0$ and take a compact neighborhood $W$ of $e\in G$ with $\|\alpha_{t^{-1}}(\phi(f))a-\phi(f)a\|\leq\epsilon$ for all $t\in W$. Now, take $h\in\contc(G)^+$ with $\supp(h)\sbe W$ and $\int_G\Delta(t)^{-1/2}h(t) dt=1$. Then, for $\xi(t)\defeq h(t)a$, we have
\begin{align*}
\|\Psi(f\cdot \xi)-f\cdot a\|&=\left\|\int_W\Delta(t)^{-1/2}\left(\alpha_{t^{-1}}(\phi(f))h(t)a -\phi(f)h(t)a\right)dt\right\|\\
 &\leq \int_W\Delta(t)^{-1/2}\|\alpha_{t^{-1}}(\phi(f))a-\phi(f)a\|h(t) dt \leq \epsilon.
\end{align*}
This finishes the proof of the first assertion in the proposition. For the second, we have to check that $\Psi(m\cdot\zeta)=m\Psi(\zeta)$ for all $m\in A^{G,\alpha}_c$ and $\zeta\in \H_c$.
But this follows from an easy computation using the fact that $m$ is $G$\nb-invariant. The final assertion follows from Proposition~\ref{prop-induced} and the fact that $\Lambda_A$ factors through a faithful representation of $A\rtimes_\red G$, so that $\Ind^{G,\alpha}_\red$ is a faithful representation of $A^{G,\alpha}_\red$. Since this representation extends the inclusion map $A^{G,\alpha}_c\into \M(A)$, the final assertion follows.
\end{proof}

\begin{remark} As a consequence of the above proposition we obtain the explicit description
of the reduced generalized fixed-point algebra $A_r^{G,\alpha}$ as
\begin{equation}\label{eq:descriptionRedGFPA}
A_r^{G,\alpha}=\contz(G\bs X)\cdot \overline{\{m\in \M(A)^{G,\alpha}: f\cdot m, m\cdot f\in A_c\;\forall f\in C_c(X)\}} \cdot\contz(G\bs X),
\end{equation}
which, as it seems, has not been obtained so far in general. For centrally proper $X\rtimes G$-algebras (\ie, if  $\phi(C_0(X))\subseteq Z\M(A)$), this easily implies the description
\begin{equation}\label{eq:descriptionRedGFPA-Kas}
A_r^{G,\alpha}=\contz(G\bs X)\cdot \{m\in \M(A)^{G,\alpha}: f\cdot m, m\cdot f\in A\;\forall f\in C_0(X)\},
\end{equation}
which is Kasparov's definition (see \cite{Kasparov:Novikov}) of the generalized fixed-point algebra in this situation.
 But in general one should be careful not to  mistake the algebra in
(\ref{eq:descriptionRedGFPA}) with the algebra
\begin{equation}\label{eq-wrong}
\contz(G\bs X)\cdot \{m\in \M(A)^{G,\alpha}: f\cdot m, m\cdot f\in A\;\forall f\in C_0(X)\} \cdot\contz(G\bs X),
\end{equation}
which, in general can be very different from $A_r^{G,\alpha}$. For example, if $G$ is a discrete group and
$A=\K(L^2G)$ is equipped with the structure of a weakly proper  $G\rtimes G$\nb-algebra with respect to the action $\alpha=\Ad\rho$ of $G$ on
$A$, the
right translation action of $G$ on $G$, and  the representation $M:C_0(G)\to \Lb(L^2G)$ as multiplication operators,
then (as we shall see in \S \ref{sec-Landstad} below) $A_r^{G,\alpha}\cong C_r^*(G)$, while the algebra
in (\ref{eq-wrong}) equals the group von Neumann algebra $\Lb(G)$ of $G$. (Use $G\bs G=\{\point\}$ together with the fact
that all multiplication operators $M(f)$ for $f\in C_0(G)$ are compact operators. It  follows that
the algebra in (\ref{eq-wrong}) is  the commutator of the right regular representation of $G$.)
\end{remark}

\section{Landstad duality for general coactions}\label{sec-Landstad}

As an application of the theory of generalized fixed-point algebras developed here, we now want to
obtain a general version of Landstad duality for coactions. As a consequence, we shall
prove some results on Katayama duality for coactions with respect to ``intermediate'' crossed products as
studied in \cite{Kaliszewski-Landstad-Quigg:Exotic}.

Our Landstad duality theorem extends Quigg's Landstad duality theorem for
reduced (or normal) coactions (see \cite{Quigg:Landstad_duality}), and the work presented in this section
is very much inspired by that result although the details of our proof differ substantially.
Notice that a version of Landstad duality for maximal coactions is also known (see \cite[Corollary~4.3]{Kaliszewski-Quigg-Raeburn:ProperActionsDuality}) and can be obtained from the version for normal coactions and the fact that taking maximalizations and normalizations yield equivalences between the categories of maximal and normal coactions (\cite[Corollary~A.3]{Kaliszewski-Quigg-Raeburn:ProperActionsDuality}). Our duality theorem will give an alternative direct proof which does not use this equivalence and also works for other exotic completions.

Recall that a coaction of a locally compact group $G$ on a \cstar{}algebra $B$ is a nondegenerate
injective \Star{}homomorphism $\delta:B\to \M(B\otimes \Cst(G))$
such that
\begin{equation}\label{eq-coaction}
(\id_B\otimes\, \delta_G)\circ \delta=(\delta\otimes \id_G)\circ \delta
\end{equation}
as maps from $B$ to $\M(B\otimes \Cst(G)\otimes \Cst(G))$,
where $\delta_G:\Cst(G)\to \M(\Cst(G)\otimes \M(\Cst(G)))$ is the integrated form of the
strictly continuous homomorphism
$G\ni s\mapsto u_s\otimes u_s \in U\M(\Cst(G)\otimes \Cst(G))$
 and $s\mapsto u_s$ denotes the canonical inclusion of $G$ into the group
$U\M(\Cst(G))$ of unitary multipliers of $\Cst(G)$.
Note that ``$\otimes$'' denotes the minimal (or spatial) \cstar{}tensor product.
We shall always assume that the coaction is \emph{nondegenerate} in the
sense that $\cspn\big(\delta(B)(1\otimes \Cst(G))\big)=B\otimes \Cst(G)$.

In what follows we write $w_G:G\to U\M(\Cst(G)) $ for the map $w_G(s)=u_s$.
Then $w_G$ may be viewed as a unitary in $\M(\contz(G)\otimes \Cst(G))\cong \contb^{\st}(G, \M(\Cst(G)))$, the \cstar{}algebra of strictly continuous bounded functions $G\to \M(\Cst(G))$, and for any
nondegenerate \Star{}homomorphism $\sigma:\contz(G)\to \M(D)$, where $D$ is some \cstar{}algebra, we can
consider the element $\sigma\otimes \id(w_G)\in U\M(D\otimes \Cst(G))$.
Recall that a \emph{covariant homomorphism} of the cosystem $(B,G,\delta)$ into $\M(D)$ consists of a pair
$(\pi,\sigma)$ in which $\pi:B\to \M(D)$ and $\sigma:\contz(G)\to \M(D)$ are nondegenerate \Star{}homomorphisms
satisfying the \emph{covariance condition}:
\begin{equation}\label{co-covariance}
(\pi\otimes \id_G)\circ \delta(b)=(\sigma\otimes \id_G)(w_G)(\pi(b)\otimes 1)(\sigma\otimes\id_G)(w_G)^*\quad\mbox{for all } b\in B.
\end{equation}
A \emph{crossed product} of $(B, G,\delta)$ is a triple $(A, j_B, j_{\contz(G)})$ consisting of a
\cstar{}algebra $A$ together with a covariant homomorphism $(j_B, j_{\contz(G)})$ of $(B,G,\delta)$
into $\M(A)$ such that
\begin{enumerate}
\item $A=\overline{j_B(B)j_{\contz(G)}(\contz(G))}\subseteq \M(A)$;
\item For any covariant homomorphism $(\pi,\sigma)$ of $(B,G,\delta)$ into some $\M(D)$ there exists
a \Star{}homomorphism $\pi\rtimes\sigma: A\to \M(D)$ such that $\pi=(\pi\rtimes \sigma)\circ j_B$ and
$\sigma=(\pi\rtimes\sigma)\circ j_{\contz(G)}$.
\end{enumerate}
A crossed product always exists and is unique up to isomorphism. We denote it by
$(B\rtimes_\delta \widehat{G}, j_B, j_{\contz(G)})$ (the notation $\widehat{G}$ indicates that a coaction crossed
product should be regarded as a crossed product by a dual object of $G$ -- indeed, in case where $G$ is abelian, it
is a crossed product by an action of the dual group $\widehat{G}$).

If $B\rtimes_\delta \widehat G$ is a coaction crossed product, there is a \emph{dual action}
$\widehat{\delta}:G\to \Aut(B\rtimes_\delta \widehat{G})$ given by $\widehat{\delta}_s=j_B\rtimes (j_{\contz(G)}\circ \r_s)$
for $s\in G$, where $\r_s\in \Aut(\contz(G))$ is the right translation action $(\r_s(f) )(t)=f(ts)$. With respect to this action,
the canonical homomorphism $j_{\contz(G)}\colon \contz(G)\to \M(B\rtimes_\delta\dualG)$ is $G$\nb-equivariant,
so that $(B\rtimes_\delta\dualG,\dual\delta,j_{\contz(G)})$ is a weakly proper $G\rtimes G$\nb-algebra.
We always endow $B\rtimes_\delta \dualG$ with this structure of a weakly proper $G\rtimes G$\nb-algebra and we shall simply say \emph{weak} $G\rtimes G$\nb-algebra below.

By Katayama's duality theorem (\cite{Katayama:Takesaki_Duality}), we always have a canonical surjective \Star{}homomorphism
\begin{equation}\label{eq-dualhom}
\Phi_B: B\rtimes_\delta \widehat{G}\rtimes_{\widehat{\delta}}G\to B\otimes \K(L^2G)
\end{equation}
which is given as the integrated form of the covariant homomorphism
$(\pi, U)$ of $(B\rtimes_\delta \widehat{G}, G,\widehat\delta)$ with
$$\pi= (\id_B\otimes\lambda)\circ \delta\rtimes (1\otimes M)\quad\text{and}\quad U=1_B\otimes \rho$$
where $\lambda$ and $\rho$ denote the left and right regular representations of $G$ and $M:\contz(G)\to \Lb(L^2G)$
the representation by multiplication operators.
For a quite detailed overview of the theory of co-systems
and their crossed products we refer to \cite{Echterhoff-Kaliszewski-Quigg-Raeburn:Categorical}.
We recall the following definition given in \cite{Echterhoff-Kaliszewski-Quigg:Maximal_Coactions}:

\begin{definition}\label{def-max-normal}
The coaction $\delta:B\to \M(B\otimes \Cst(G))$ is called \emph{maximal} if the map $\Phi_B$
in (\ref{eq-dualhom})
is an isomorphism. Moreover, the coaction $\delta$ is called \emph{normal} if $\Phi_B$ factors through
an isomorphism $\Phi_{B,r}:B\rtimes_\delta \widehat{G}\rtimes_{\widehat\delta,r}G\to B\otimes \K(L^2G)$.
\end{definition}

Normal coactions have first been studied by Quigg in \cite{Quigg:FullAndReducedCoactions}. He shows
that a coaction $\delta$ is normal if and only if $j_B:B\to \M(B\rtimes_\delta \widehat{G})$ is injective
and that every coaction has a \emph{normalization} which is given by the coaction
$\delta_\red: B_\red\to \M(B_\red\otimes \Cst(G))$ on the quotient $B_\red\defeq B/I$, where $I=\ker j_B$, given by the formula
$\delta_\red(b+I):=(q\otimes \id_G)\circ \delta(b)$, where $q:B\to B_\red$ denotes the quotient map. This homomorphism induces an isomorphism $B\rtimes_\delta\dualG\congto B_\red\rtimes_{\delta_\red}\dualG$ of weak $G\rtimes G$\nb-algebras.
Moreover, if $\|\cdot\|_\mu$ is any crossed-product norm on $C_c(G,A)$ for a system $(A,G,\alpha)$ such
that $A\rtimes_{\alpha,\mu}G$ admits a dual coaction $\widehat{\alpha}_\mu$,  the dual coaction $\widehat{\alpha}_r$
on the reduced crossed product $A\rtimes_{\alpha,r}G$ is the normalization of $\widehat{\alpha}_\mu$.

On the other hand, it is shown in \cite{Echterhoff-Kaliszewski-Quigg:Maximal_Coactions} that, for every coaction $\delta$, there exists a maximal
coaction $\delta_{u}: B_{u}\to \M(B_u\otimes \Cst(G))$ such that $B=B_u/J$ for a suitable ideal
$J$ in such a way that $\delta_u$ factors through $B$ to give the original coaction $\delta$ and such that the surjection $B_u\onto B$ induces an isomorphism $B_u\rtimes_{\delta_u}\widehat{G}\cong B\rtimes_{\delta}\widehat{G}$ of weak $G\rtimes G$\nb-algebras. The coaction $(B_u,\delta_u)$ is, up to isomorphism, uniquely determined by these properties and is called the \emph{maximalization} of $(B,\delta)$. The construction of maximalizations given in \cite{Echterhoff-Kaliszewski-Quigg:Maximal_Coactions} is quite involved and is not functorial. Our results below will give, in particular, an alternative functorial construction for maximalizations of arbitrary coactions. Note that the dual coaction $\widehat{\alpha}_u$ on the full crossed product $A\rtimes_{\alpha}G$  is the maximalization of a dual coaction
$\widehat{\alpha}_\mu$ on an intermediate crossed product $A\rtimes_{\alpha,\mu}G$.

\begin{definition}\label{def-E-dual}
Let $\delta:B\to \M(B\otimes \Cst(G))$ be a coaction and let $\|\cdot\|_\pn$ be any crossed-product norm on
$\contc(G, B\rtimes_\delta \widehat{G})$, viewed as subalgebra of $(B\rtimes_\delta \widehat{G})\rtimes_{\widehat\delta}G$.
Then we say that $\delta$ satisfies \emph{$\pn$-duality} or, shortly, that $\delta$ is a \emph{$\pn$\nb-coaction}, if the homomorphism
$\Phi_B$ of (\ref{eq-dualhom}) factors through an isomorphism
$$\Phi_{B,\pn}:B\rtimes_\delta \widehat{G}\rtimes_{\widehat{\delta}, \pn}G\to B\otimes \K(L^2G).$$
A {\em $\pn$-ization} of $(B,\delta)$ is a $G$\nb-coaction $(B_\pn,\delta_\pn)$ together with an isomorphism $B\rtimes_\delta\dualG\cong B_\pn\rtimes_{\delta_\pn}\dualG$ of weak $G\rtimes G$\nb-algebras in such a way that viewing $\|\cdot \|_\pn$ as a crossed-product norm on $\contc(G,B_\pn\rtimes_{\delta_\pn}\dualG)$ via this isomorphism, $(B_\pn,\delta_\pn)$ is a $\pn$\nb-coaction.
\end{definition}

Notice that in the above definition we do not require that $B\rtimes_\delta \widehat{G}\rtimes_{\widehat{\delta}, \pn}G$ carries a (bi)dual coaction---in fact the discussion below shows that
this is automatic. For the reduced crossed-product norms, the above definition specializes to normal coactions and
normalizations, \ie, an $\red$-coaction is just a normal coaction and an $\red$-ization is just a normalization.
Similarly, for maximal crossed-product norms we get maximal coactions and maximalizations, \ie, a $\un$-coaction is a maximal coaction and a $\un$-ization is a
maximalization of a given coaction.

It is clear that every coaction satisfies $\pn$-duality for some crossed-product norm $\|\cdot\|_\pn$, since
the quotient $\big(B\rtimes_\delta \widehat{G}\rtimes_{\widehat\delta}G\big)/\ker\Phi_B$ always lies between
the maximal and the reduced crossed product by $\widehat{\delta}$.
On the other hand, if $\|\cdot\|_\pn$ is some crossed-product norm on $\contc(G, B\rtimes_\delta \widehat{G})$
such that $(B,\delta)$ satisfies $\pn$-duality, then the canonical $G$\nb-coaction on $B\otimes \K(L^2G)$ (see Equation~\eqref{eq:CoactionOnBotimesK}) may be viewed as a coaction for the double crossed product $B\rtimes_\delta\dualG\rtimes_{\dual\delta,\pn}G\cong B\otimes \K(L^2G)$. This coaction necessarily factors the bidual coaction on $B\rtimes_\delta\dualG\rtimes_{\dual\delta,\un}G$.
Conversely, we shall see below
that for any given crossed product $B\rtimes_\delta \widehat{G}\rtimes_{\widehat{\delta}, \pn}G$ admitting
a (bi)dual coaction $\widehat{\widehat{\delta}}_\pn$, there exists a $\pn$-ization $(B_\pn,\delta_\pn)$ of $(B,\delta)$.
This will lead, in particular, to a positive answer of Conjecture 6.14 in \cite{Kaliszewski-Landstad-Quigg:Exotic}.

The main result in this section
is the following theorem, which provides a general version
of Landstad $\pn$-duality for coactions. Note that in case of normal coactions,
Landstad duality has been obtained by Quigg in \cite{Quigg:Landstad_duality} (see
also \cite{Kaliszewski-Quigg-Raeburn:ProperActionsDuality, Kaliszewski-Quigg:Categorical_Landstad}).

In what follows, $\contz(G)$ will be always endowed with the right translation action $\r:G\to\Aut(C_0(G))$.

\begin{theorem}[{\cf~\cite[Theorem 3.3]{Quigg:Landstad_duality}}]\label{thm-landstad}
Suppose that $A$ is a weakly proper $G\rtimes G$\nb-algebra with respect to the $G$\nb-equivariant
 structure map $\phi: \contz(G)\to \M(A)$ and the action $\alpha:G\to \Aut(A)$.
 Assume that $\|\cdot \|_\pn$ is a crossed-product norm on $\contc(G,A)$ which lies between
 $\|\cdot \|_u$ and $\|\cdot\|_r$ and which admits a dual coaction
 $$\widehat{\alpha}_\pn:A\rtimes_{\alpha,\pn}G\to \M(A\rtimes_{\alpha,\pn}G\otimes \Cst(G)).$$
 Let us write $B_\pn:=A_\pn^{G,\alpha}$ for any such $\|\cdot\|_\pn$.
 Then there is a canonical coaction $\delta_{\F_\pn(A)}$ of $G$ on $\F_\pn(A)$
\textup(described in Lemma \ref{lem-wg} below\textup)
 which is compatible with the dual coaction $\widehat{\alpha}_\pn$ on $A\rtimes_{\alpha,\pn}G$ and
therefore induces a compatible coaction $\delta_\pn$ on
$B_\pn\cong \K(\F_\pn(A))$ such that the following are true:
\begin{enumerate}
 \item The cosystem $(B_\pn,\delta_\pn)$ is Morita equivalent to $(A\rtimes_{\alpha,\pn}G, \widehat{\alpha}_\pn)$.
 \item $(B_u,\delta_u)$ is the maximalization of $(B_\pn,\delta_\pn)$ and $(B_r, \delta_r)$ is the normalization
 of $(B_\pn,\delta_\pn)$.
 \item The dual system $(B_\pn\rtimes_{\delta_\pn}\widehat{G}, G, \widehat{\delta}_\pn)$ is isomorphic to
$(A, G,\alpha)$ as $G\rtimes G$\nb-algebras via the covariant homomorphism
$k\rtimes\phi: B_\pn\rtimes_{\delta_\pn}\widehat{G}\to A$, where
 $k:B_\pn\to \M(A)$ extends the canonical inclusion $A_c^{G,\alpha}\into \M(A)$.
 \item $(B_\pn, G,\delta_\pn)$ satisfies $\pn$-duality and hence is a $\pn$-coaction.
 \end{enumerate}
Conversely, let $(B,\delta)$ be a $\pn$\nb-coaction for some crossed-product norm $\|\cdot\|_\pn$ on $\contc(G,B\rtimes_\delta\dualG)$
 and let $B_\pn=A_\pn^{G,\alpha}$ for the weak $G\rtimes G$\nb-algebra $(A,\alpha)\defeq (B\rtimes_\delta\dualG,\dual\delta)$
 equipped with the coaction $\delta_{\pn}$. Then $(B_\pn,\delta_\pn)$ is isomorphic to $(B,\delta)$.
\end{theorem}

Applying Theorem \ref{thm-landstad} to the weak $G\rtimes G$\nb-algebra $(B\rtimes_\delta \dualG, \dual\delta)$ for a given coaction
$\delta:B\to \M(B\otimes C^*(G))$ gives:

\begin{corollary}\label{cor-landstad}
Suppose that $\|\cdot\|_\pn$ is any crossed-product norm on $\contc(G, B\rtimes_\delta \widehat{G})$ such that
the crossed product $B\rtimes_\delta \widehat{G}\rtimes_{\widehat\delta, \pn}G$ admits a dual coaction.
Then $(B,\delta)$ admits a $\pn$-ization, \ie, there exists a coaction $\delta_\pn:B_\pn\to \M(B_\pn\otimes \Cst(G))$
``lying between'' the maximalization $\delta_u$ and the normalization $\delta_r$ of $\delta$
such that $\delta_\pn$ satisfies $\pn$-duality and such that the dual systems
$(B_\pn\rtimes_{\delta_\pn}\widehat{G}, G, \widehat{\delta}_\pn)$ and $(B\rtimes_\delta \widehat{G}, G, \widehat{\delta})$
are isomorphic as weak $G\rtimes G$\nb-algebras. If $\delta:B\to\M(B\otimes C^*(G))$ is already a $\pn$\nb-coaction,
then $(B_\pn,\delta_\pn)\cong (B,\delta)$.
\end{corollary}

Recall that for a right Hilbert $B$-module $\E$, the multiplier module $\M(\E)$ is defined as
the set $\L_B(B,\E)$ of adjointable operators from the standard Hilbert $B$-module $B$ into $\E$.
The $B$-valued inner product on $\E$ then extends to an $\M(B)$-valued inner product on
$\M(\E)$ by the formula $\braket{m}{n}_{\M(B)}=m^*\circ m\in \L_B(B)=\M(B)$.
On the other hand, we get a left $\M(\K_B(\E))$-valued inner product by the
formula ${_{\M(\K(\E))}\braket{m}{n}}=m \circ n^*\in \L_B(\E)=\M(\K_B(\E))$.
Note that there is a canonical inclusion map $\E\into \M(\E)$ by identifying an element
$\xi\in \E$ with the operator $b\mapsto \xi\cdot b\in \L_B(B,\E)$. The adjoint of this operator
is given by $\eta\mapsto \braket{\xi}{\eta}_B\in \L_B(\E,B)$.
For more information on multiplier bimodules, see \cite{Echterhoff-Kaliszewski-Quigg-Raeburn:Categorical}.

In what follows we write $\E\otimes D$ for the Hilbert $B\otimes D$-module
which is obtained as the external (minimal) tensor product
of the Hilbert $B$-module $\E$ with the \cstar{}algebra $D$, viewed as a Hilbert $D$-module.

To prove Theorem~\ref{thm-landstad}, we start with the following preliminary result:

\begin{lemma}\label{lem-tensor}
Suppose that $A$ is a weakly proper $X\rtimes G$-algebra with respect to the action
$\alpha:G\to \Aut(A)$ and the structure map $\phi:\contz(X)\to \M(A)$. Let $D$ be any
\cstar{}algebra and consider $A\otimes D$ as a weakly proper $X\rtimes G$-algebra
for the action $\alpha\otimes\id_D$ and the structure map $\phi\otimes 1$.

Then, for each given crossed-product norm $\|\cdot\|_\pn$ on $\contc(G,A)$, there exists a \textup(unique\textup)
corresponding crossed-product norm $\|\cdot \|_\qn$ on $\contc(G, A\otimes D)$
such that the canonical inclusion $\contc(G,A)\odot D\subseteq \contc(G, A\otimes D)$ induces
an isomorphism
$$(A\rtimes_{\alpha,\pn}G)\otimes \Cst(G)\cong (A\otimes D)\rtimes_{\alpha\otimes \id, \qn}G.$$
The corresponding inclusion $\F_c(A)\otimes D\subseteq \F_c(A\otimes D)$ extends to
an isomorphism
$$\F_\pn(A)\otimes D\cong \F_\qn(A\otimes D)$$
and, therefore, the inclusion of $A_c^{G,\alpha}\odot D$ into $(A\otimes D)^{G,\alpha\otimes \id_D}_c$
induces an isomorphism
$$A_\pn^{G,\alpha}\otimes D\cong (A\otimes D)^{G,\alpha\otimes \id}_\qn.$$
In particular, we may regard $\F_c(A\otimes D)$ as a dense submodule of $\F_\pn(A)\otimes D$ and
$(A\otimes D)^{G,\alpha\otimes \id_D}_c$ as a dense subalgebra of $A_\pn^{G,\alpha}\otimes D$.
\end{lemma}

\begin{remark}
If $\|\cdot\|_\pn=\|\cdot\|_r$ is the reduced norm, then $\|\cdot\|_\qn=\|\cdot\|_r$ is also the reduced norm
on $\contc(G, A\otimes D)$, which follows from the well-known isomorphism
\begin{equation}\label{eq-isoreduced}
(A\otimes D)\rtimes_{\alpha\otimes\id_D, r}G\cong (A\rtimes_{\alpha,r}G)\otimes D.
\end{equation}
But we should point out that even if $\|\cdot\|_\pn=\|\cdot \|_u$ is the universal norm, we cannot
expect in general that $\|\cdot\|_\qn$ is the universal norm as well. To see this consider the
case where $A=\K(L^2G)$ with action $\alpha=\Ad\rho$ and structure map $M:\contz(G)\to \Lb(L^2G)$.
Since $\alpha=\Ad\rho$ is implemented by a unitary representation, it is exterior equivalent to the trivial
action of $G$ on $\K(L^2G)$. Thus we get
$$
(\K(L^2G)\otimes \Cst(G))\rtimes_{\Ad\rho\otimes \id,u}G
\cong (\K(L^2G)\otimes \Cst(G))\otimes_{\max}\Cst(G)
$$
while on the other side we have
$$(\K(L^2G)\rtimes_{\Ad\rho,u}G)\otimes \Cst(G)\cong (\K(L^2G)\otimes \Cst(G))\otimes \Cst(G).$$
These algebras will be completions by different norms if the canonical quotient map
$P: \Cst(G)\otimes_{\max}\Cst(G)\to \Cst(G)\otimes \Cst(G)$ is not an isomorphism,
which is true for any group without Kirchberg's factorization property (F) (see \cite[\S 7]{Kirchberg:Nonsemisplit}).
Now, due to the work of Kirchberg and others, we know that there exist many groups which
do not satisfy this property (see \cite{Anantharaman-Delaroche:On_tensor_products} for a survey on this property).
\end{remark}
\begin{proof}[Proof of the lemma]
Let $(\iota_A, \iota_G):(A,G) \to \M(A\rtimes_{\alpha}G)$ denote the canonical inclusions and
let $j_{A\rtimes G}$ and $j_D$ denote the inclusions of $A\rtimes_{\alpha}G$ and
$D$ into $\M(A\rtimes_{\alpha}G\otimes D)$, respectively.
Then $((j_{A\rtimes G}\circ \iota_A)\otimes j_{D}, \iota_G\otimes 1_D)$ is a covariant homomorphism
of $(A\otimes D, G, \alpha\otimes \id)$ into $\M(A\rtimes_{\alpha}G\otimes D)$ whose
integrated form $\Phi_{u}$ restricts to the inclusion $\contc(G,A)\odot D\into \contc(G, A\otimes D)$.
Thus we see that
$$\Phi_{u}: (A\otimes D)\rtimes_{\alpha\otimes\id_D, u}G\to (A\rtimes_{\alpha}G)\otimes D$$
is a surjective \Star{}homomorphism.
Since
$A\rtimes_{\alpha,\pn}G$ ``lies'' between $A\rtimes_{\alpha, u}G$ and $A\rtimes_{\alpha, r}G$,
we conclude from this and (\ref{eq-isoreduced}) that $(A\rtimes_{\alpha,\pn}G)\otimes D$ ``lies'' between
$(A\otimes D)\rtimes_{\alpha\otimes\id_D, u}G$ and $(A\otimes D)\rtimes_{\alpha\otimes\id_D, r}G$
and it follows that $(A\rtimes_{\alpha,\pn}G)\otimes D\cong (A\otimes D)\rtimes_{\alpha\otimes\id,\qn}G$
for a suitable crossed-product norm $\|\cdot\|_\qn$.

Now, since $A\odot D$ is norm dense in $A\otimes D$, it follows that $\F_c(A)\odot D= \contc(X)A\odot D$
is inductive limit dense in $\F_c(A\otimes D)$, and since the $A\rtimes_{\alpha,\pn}G\otimes D$-valued
inner product on $\F_c(A)\odot D$ coincides with the $(A\otimes D)\rtimes_{\alpha\otimes \id_D, \qn}G$-valued inner
product by the choice of $\qn$, we see that $\F_\pn(A)\otimes D\cong \F_\qn(A\otimes D)$.
 The result follows.
\end{proof}

We also need the following (certainly well-known) auxiliary result:

\begin{lemma}\label{lem-max}
Suppose that $\E$ is a $B-C$ imprimitivity bimodule and that $\delta_\E$ is a coaction of $G$ on $\E$
which implements a Morita equivalence between coactions $(B, \delta_B)$ and $(C,\delta_C)$ of $G$.
Then $\delta_B$ is maximal (resp. normal) if and only if $\delta_C$ is maximal (resp. normal)
and the coaction $(\E, \delta_\E)$ factors through a Morita equivalence
$(\E_n,\delta_{\E_n})$ between the normalizations $(B_n, \delta_{B_n})$ and $(C_n,\delta_{C_n})$.
\end{lemma}
\begin{proof} This follows from an easy linking algebra argument which we omit.
\end{proof}

In the following lemma we allow a slightly more general situation than what we really need in this section, namely
we assume that a closed subgroup $H$ of $G$ is given and consider a $G\rtimes H$-algebra $A$, \ie, a \cstar{}algebra $A$ endowed with an $H$-action $\alpha$ and an $H$-equivariant nondegenerate \Star{}homomorphism $\phi\colon \contz(G)\to \M(A)$ (where $\contz(G)$ is now endowed with right translation $H$-action $\r$). In addition, we assume that $\|\cdot\|_\pn$ is a \cstar{}norm on $\contc(H,A)$ (between $\|\cdot \|_u$ and $\|\cdot \|_\red$) such that the dual coaction $\widehat{\alpha}: A\rtimes_{\alpha,u} H\to \M(A\rtimes_{\alpha,u}H\otimes \Cst(H))$
factors through a coaction
$$\widehat{\alpha}_\pn:A\rtimes_{\alpha,\pn}H\to \M(A\rtimes_{\alpha,\pn}H\otimes \Cst(H)).$$
We then consider the \emph{inflated} $G$\nb-coaction $\Inf\widehat{\alpha}_\pn\defeq (\Id\otimes \iota_{G,H})\circ\widehat{\alpha}_\pn\colon A\rtimes _{\alpha,\pn}H\to \M(A\rtimes_{\alpha,\pn}H\otimes \Cst(G))$, where $\iota_{G,H}\colon \Cst(H)\to \M(\Cst(G))$ denotes the canonical homomorphism defined as the integrated form of the obvious representation $H\to U\M(\Cst(G))$ sending $t\in H$ to $u_t\in U\M(\Cst(G))$ (see \cite{Echterhoff-Kaliszewski-Quigg-Raeburn:Categorical} for more information on inflated coactions).
This extra generality will be used in \S 5 of the forthcoming paper \cite{Buss-Echterhoff:Mansfield}.

\begin{lemma}\label{lem-wg}
Let $(A,\alpha)$ be a $G\rtimes H$-algebra as above, and let $\F_\pn(A)$ denote the corresponding Hilbert $A\rtimes_{\alpha,\pn}H$-module.
Then there is a canonical Hilbert-module coaction $\delta_{\F_\pn(A)}:\F_\pn(A)\to \M(\F_\pn(A)\otimes \Cst(G))$ compatible
with $\Inf\widehat{\alpha}_\pn$ which is given, for $\xi\in \F_c(A)=\phi(\contc(G))A$, by the formula
$$\delta_{\F_\pn(A)}(\xi)=\phi\otimes \id_G(w_G)(\xi\otimes 1).$$
Moreover, via the isomorphism $\F_\pn(A)\cong \F(G)\otimes_{\contz(G)\rtimes H} A\rtimes_{\alpha,\pn}H$ of Proposition~\ref{prop-factor},
the coaction $\delta_{\F_\pn(A)}$ corresponds to the \textup(balanced\textup) tensor product \textup(as defined in \cite{Echterhoff-Kaliszewski-Quigg-Raeburn:Categorical}*{Proposition~2.13}\textup) of the coactions $\delta_{\F(G)}$ and $\Inf\dual\alpha$, where $\delta_{\F(G)}$ denotes the coaction of $G$ on $\F(G)=\F(\contz(G))$ given by $\delta_{\F(G)}(f)=\omega_G(f\otimes 1)$ for all $f\in \contc(G)$.
\end{lemma}
\begin{proof}
We first need to check that the right hand side of the equation makes sense, \ie, that
$\phi\otimes \id_G(w_G)(\xi\otimes 1)$ determines an adjointable operator from
$A\rtimes_{\alpha,\pn}H\otimes \Cst(G)$ to $\F_\pn(A)\otimes \Cst(G)$.
For this we first observe that for any $z\in \Cst(G)$ we get (with $\pn$ and $\qn$ as in Lemma~\ref{lem-tensor} applied to $X=G$ and $H$ in place of $G$):
$$\phi\otimes \id_G(w_G)(\xi\otimes z)\in \F_c(A\otimes \Cst(G))\subseteq \F_\qn(A\otimes \Cst(G))\cong \F_\pn(A)\otimes \Cst(G).$$
Indeed, writing $\xi=\phi(f) a$ with $f\in \contc(X)$ and $a\in A$, we get
\begin{equation}\label{eq-FAc}
\phi\otimes \id_G(w_G)(\xi\otimes z)=\phi\otimes \id_G(w_G(f\otimes z))(a\otimes 1)\in \F_c(A\otimes \Cst(G)),
\end{equation}
since $w_G(f\otimes z)\in \contc(G, \Cst(G))$. Suppose now that $w\in A\rtimes_{\alpha,\pn}H\otimes \Cst(G)$.
Factorizing $w=(1\otimes z )w'$ with $z\in \Cst(G)$ and $w'\in A\rtimes_{\alpha,\pn}H\otimes \Cst(G)$, gives
$$\delta_{\F_\pn(A)}(\xi)\cdot w=\phi\otimes \id_G(w_G)(\xi\otimes z)\cdot w'\in \F_\pn(A)\otimes \Cst(G).$$
It is straightforward to check that this does not depend on the
given factorization $w=(1\otimes z)w'$ and that $w\mapsto \delta_{\F_\pn(A)}(\xi)\cdot w$ is
adjointable with adjoint given by the formula
$$\delta_{\F_\pn(A)}(\xi)^*(\eta)=\bbraket{\phi\otimes \id_G((1\otimes v^*)w_G)(\xi\otimes 1)}{\eta'}_{A\rtimes_{\alpha,\pn}H\otimes \Cst(G)}$$
if we factorize $\eta\in \F_\pn(A)\otimes \Cst(G)$ as $\eta=(1\otimes v)\eta'$ for some $v\in \Cst(G)$. One
may check, as above, that $\phi\otimes \id_G((1\otimes v^*)w_G)(\xi\otimes 1)\in \F_c(A\otimes \Cst(G))$, so
that the inner product makes sense and gives an element in $(A\otimes \Cst(G))\rtimes_{\alpha\otimes\Id,\qn}H\cong (A\rtimes_{\alpha,\pn}H)\otimes \Cst(G)$.

Note that it follows easily from (\ref{eq-FAc}) that
$\delta_{\F_\pn(A)}(\F_c(A))(1\otimes \Cst(G))$ is dense in $\F_c(A\otimes \Cst(G))$ in the inductive limit topology
and hence is dense in $\F_\pn(A)\otimes \Cst(G)$. We now compute, for all $z,v\in \Cst(G)$,
$\xi=\phi(f)a$, $\eta=\phi(g)b\in \F_c(A)$ and $t\in H$:
\begin{align*}
&(1\otimes z^*)\bbraket{\delta_{\F_\pn(A)}(\xi)}{\delta_{\F_\pn(A)}(\eta)}_{\M(A\rtimes_\pn G\otimes \Cst(G))}(1\otimes v)|_t\\
&=\bbraket{\delta_{\F_\pn(A)}(\xi)(1\otimes z)}{\delta_{\F_\pn(A)}(\eta)(1\otimes v)}_{A\rtimes_\pn G\otimes \Cst(G)}|_t\\
&=\Delta(t)^{-1/2} \big(\delta_{\F_\pn(A)}(\xi)(1\otimes z)\big)^*\alpha_t\otimes \id_G\big(\delta_{\F_\pn(A)}(\eta)(1\otimes v)\big)\\
&=\Delta(t)^{-1/2} \Big(\phi\otimes \id_G(w_G)(\xi\otimes z)\Big)^*\alpha_t\otimes \id_G\Big(\phi\otimes \id_G(w_G)(\eta\otimes v)\Big)\\
&=\Delta(t)^{-1/2}(a^*\otimes z^*)\Big(\phi\otimes \id_G\big((\bar{f}\otimes 1)w_G^*\big)\Big)\alpha_t\otimes \id_G\Big(\phi\otimes \id_G\big(w_G(g\otimes 1)\big)(b\otimes v)\Big)\\
&=\Delta(t)^{-1/2}(a^*\otimes z^*)\bigg(\phi\otimes \id_G\Big(\big((\bar{f}\otimes 1)w_G^*\big)\r_t\otimes \id_G\big(w_G(g\otimes 1)\big)\Big)\bigg)(\alpha_t(b)\otimes v)=...
\end{align*}
Now observe that the middle part $\big((\bar{f}\otimes 1)w_G^*\big)\r_t\otimes \id_G
\big(w_G(g\otimes 1)\big)$ is the function in $\contc^{\st}(G, \M(\Cst(G)))$ given by
$$s\mapsto \bar{f}(s)w_G^*(s)w_G(st) g(st)=\bar{f}(s) u_s^* u_{st} g(st)=(\bar{f}\r_t(g))(s)u_t.$$
Hence we can proceed the above computation with
\begin{align*}
&...= \Delta(t)^{-1/2}(1\otimes z^*)\big((a^*\cdot \bar{f})\alpha_t(g\cdot b)\otimes u_t\big)(1\otimes v)\\
&=(1\otimes z^*)\big(\bbraket{\xi}{\eta}_{A\rtimes_{\pn}G}(t) \otimes u_t\big)(1\otimes v)\\
&=(1\otimes z^*)\Inf\widehat{\alpha}_\pn\big(\bbraket{\xi}{\eta}_{A\rtimes_\pn G}\big)(1\otimes v)|_t.
\end{align*}
Since $z,v\in \Cst(G)$ have been arbitrary, we see that
$$\bbraket{\delta_{\F_\pn(A)}(\xi)}{\delta_{\F_\pn(A)}(\eta)}_{\M(A\rtimes_\pn G\otimes \Cst(G))}=
\Inf\widehat{\alpha}_\pn(\bbraket{\xi}{\eta}_{A\rtimes_\pn G})$$
for all $\xi,\eta\in \F_c(A)$. Since $\Inf\widehat{\alpha}_\pn$ is isometric, the same follows
for $\delta_{\F_\pn(A)}$ with respect to the norm on $\F_\pn(A)$, so we see that
$\delta_{\F_\pn(A)}$ extends uniquely to an isometric map
$$\delta_{\F_\pn(A)}: \F_\pn(A)\to \M(\F_\pn(A)\otimes \Cst(G))$$
such that $\cspn\big(\delta_{\F_\pn(A)}(\F_\pn(A))(1\otimes \Cst(G))\big)=\F_\pn(A)\otimes \Cst(G)$.

In order to complete the proof of the first part of the lemma, we only need to show that $\delta_{\F_\pn(A)}$ satisfies
the coaction identity
$$(\delta_{\F_\pn(A)}\otimes \id_G)\circ \delta_{\F_\pn(A)}=(\id_{\F_\pn(A)}\otimes \delta_G)\circ \delta_{\F_\pn(A)}.$$
For this let $\xi=\phi(f)a\in \F_c(A)$ and $z\in C^*(G)$. As explained above, the element
$$x\defeq \delta_\F(\xi)(1\otimes z)=\phi\otimes\id(\omega_G)(\xi\otimes z)\in \F_c(A\otimes C^*(G))\sbe \F_\qn(A\otimes C^*(G))$$
is viewed as an element in $\F_\pn(A)\otimes C^*(G)$ via the canonical isomorphism $\F_\qn(A\otimes C^*(G))\cong \F_\pn(A)\otimes C^*(G)$. For such elements we have
$$(\delta_{\F_\pn(A)}\otimes\id_G)(x)=(\phi\otimes\id_G(\omega_G)\otimes 1)\big(\id_\F\otimes \sigma(x\otimes 1)\big),$$
where $\sigma$ denotes the flip map on $C^*(G)\otimes C^*(G)$ and we use $\phi\otimes\id(\omega_G)(\xi\otimes z)$ in $\F_c(A\otimes C^*(G))\sbe \M(A\otimes C^*(G))$. Indeed, this assertion follows from continuity of the involved maps and the fact that $x$ can be approximated, in the inductive limit topology, by elementary tensors of the form $\eta\otimes y\in \F_c(A)\otimes C^*(G)$. Moreover, using the relation $(\id_{C_0(G)} \otimes\delta_G)(\omega_G)=(\omega_G\otimes 1)\big(\id_{C_0(G)}\otimes\sigma(\omega_G\otimes 1)\big)$, we obtain
\begin{align*}
(\delta_{\F_\pn(A)}&\otimes\id_G)\big(\delta_{\F_\pn(A)}(\xi)(1\otimes z)\big)\\
&=\big(\phi\otimes\id_G(w_G)\otimes 1\big) \Big(\id_A \otimes \sigma\big(\phi\otimes\id_G(w_G)(\xi\otimes z\otimes 1)\big)\Big)\\
&=\big(\phi\otimes\id_G(w_G)\otimes 1\big)\Big(\id_A \otimes \sigma\big(\phi\otimes\id_G(w_G)\big)\Big)(\xi\otimes 1\otimes z)\\
&=\big(\phi\otimes\id_G(w_G)\otimes 1\big)\Big(\phi\otimes\id_G\otimes\id_G\big(\id_{C_0(G)}\otimes \sigma(\omega_G\otimes 1)\big)\Big)(\xi\otimes 1\otimes z)\\
&=\big(\phi\otimes\id_G\otimes \id_G\big)\big((\omega_G\otimes 1)\big(\id_{C_0(G)}\otimes \sigma(\omega_G\otimes 1)\big)\big)(\xi\otimes 1\otimes z)\\
&=\big(\phi\otimes\id_G\otimes \id_G\big)\big(\id_{C_0(G)}\otimes \delta_G(\omega_G)\big)(\xi\otimes 1\otimes z)\\
&=\big(\id_\F\otimes\delta_G\big)\big(\phi\otimes\id_G(\omega_G)\big)(\xi\otimes 1\otimes z)\\
&=\big(\id_\F\otimes\delta_G\big)\big(\delta_\F(\xi)(1\otimes z)\big).
\end{align*}
Since $z\in C^*(G)$ was arbitrary, this gives the co-associativity of $\delta_\F$ and hence completes the proof of the first part of the lemma. For the final part, let us denote by
$\Psi_\pn\colon \F(G)\otimes_{\contz(G)\rtimes H} A\rtimes_{\alpha,\pn}H\congto \F_\pn(A)$ the isomorphism of Proposition~\ref{prop-factor} given by the formula $\Psi_\pn(f\otimes\varphi)=\int_H\Delta_H(t)^{-1/2}\alpha_t(f\cdot \varphi(t^{-1}))\dd{t}$ for all $f\in \contc(G)$ and $\varphi\in\contc(H,A)$. The tensor product of the coactions $\delta_{\F(G)}$ and $\Inf\dual\alpha_\pn$ will be denoted by $\tilde\delta$.
It is given by $\tilde\delta(f\otimes \varphi)=\theta(\delta_{\F(G)}(f)\otimes \Inf\dual\alpha_\pn(\varphi))$, where
\begin{multline*}
\theta\colon \big(\F(G)\otimes C^*(G)\big)\otimes_{\contz(G)\rtimes H\otimes C^*(G)}\big(A\rtimes_{\alpha,\pn}H\otimes C^*(G)\big)\\
\congto \big(\F(G)\otimes_{\contz(G)\rtimes H}A\rtimes_{\alpha,\pn}H\big)\otimes C^*(G)
\end{multline*}
denotes the canonical isomorphism. We have
\begin{align*}
(\Psi_\pn\otimes\id)\big(&\tilde\delta(f\otimes\varphi)\big)\\
    &=\int_H\Delta_H(t)^{-1/2}(\alpha_t\otimes\id)\big((\phi\otimes\id)(\omega_G(f\otimes 1))(\varphi(t^{-1})\otimes u_{t^{-1}})\big)\dd{t}\\
    &=\int_H\Delta_H(t)^{-1/2}(\phi\otimes\id)(\omega_G)(\phi(\tau_t(f))\otimes u_t)(\alpha_t(\varphi(t^{-1}))\otimes u_{t^{-1}})\dd{t}\\
    &=\int_H\Delta_H(t)^{-1/2}(\phi\otimes\id)(\omega_G)(\alpha_t(f\cdot\varphi(t^{-1}))\otimes 1)\dd{t}\\
    &=\delta_{\F_\pn(A)}\big(\Psi_\pn(f\otimes\varphi)\big).
\end{align*}
This shows that $\Psi_\pn$ is equivariant with respect to the coactions $\tilde\delta$ and $\delta_{\F_\pn(A)}$ and hence finishes the proof of the last statement in the lemma.
\end{proof}

\begin{remark}\label{rem-formulacoact}
It is useful to obtain an explicit formula for the coaction $\delta_\pn:B_\pn\to \M(B_\pn\otimes \Cst(G))$
on $B_\pn:=A_\pn^{H,\alpha}$ which is determined by the coaction $\delta_{\F_\pn(A)}$ of Lemma \ref{lem-wg}.
We claim that it is given on the dense subalgebra $A_c^{H,\alpha}$ by the formula
\begin{equation}\label{eq-deltap}
\delta_\pn(m)=\phi\otimes \id_G(w_G) (m\otimes 1) \phi \otimes \id_G(w_G^*)
\end{equation}
where we perform the formal computation inside $\M(A\otimes \Cst(G))$ but the outcome
can be regarded as an element in $\M(B_\pn\otimes \Cst(G))$ since $\delta_\pn(m)$ (for $m\in A_c^{H,\alpha}$)
multiplies elements
of $1\otimes \Cst(G)$ into $(A\otimes \Cst(G))^{G,\alpha\otimes \id}_c\subseteq B_\pn\otimes \Cst(G)$
(use Lemma \ref{lem-tensor}).
For a proof of (\ref{eq-deltap}), recall from Lemma \ref{lem-left-inner-product} that $A_c^{H,\alpha}=\EE(A_c)$ (where $\EE(a)=\int_H^{\st}\alpha_t(a)\dd{t}$)
with $A_c=\contc(G)\cdot A\cdot \contc(G)=\F_c(A)\cdot \F_c(A)^*$. Hence we get
$$A_c^{H,\alpha}=\EE(\F_c(A)\cdot \F_c(A)^*)={_{A_c^H}\bbraket{\F_c(A)}{\F_c(A)}}.$$
Thus we find $\xi,\eta\in \F_c(A)$ such that $m=\EE(\xi\eta^*)={_{A_\pn^G}\bbraket{\xi}{\eta}}$ and hence
\begin{align*}
\delta_\pn(m)&=\delta_\pn\big({_{A_\pn^G}\bbraket{\xi}{\eta}}\big)
={_{\M(A_\pn^G\otimes \Cst(G))}\bbraket{\delta_{\F_\pn(A)}(\xi)}{\delta_{\F_\pn(A)}(\eta)}}\\
&={_{\M(A_\pn^G\otimes \Cst(G))}\bbraket{\phi\otimes \id_G(w_G)(\xi\otimes 1)}{\phi\otimes \id_G(w_G)(\eta\otimes 1)}}\\
&=\int_H^{\st} \alpha_t\otimes \id_G\big(\phi\otimes \id_G(w_G)(\xi\eta^*\otimes 1)\phi\otimes \id_G(w_G^*)\big)\,dt.
\end{align*}
For fixed $t\in H$, we have $ \alpha_t\otimes \id_G\big(\phi\otimes \id_G(w_G)\big)=\phi\otimes \id_G\big(\r_t\otimes\id_G(w_G)\big)$ and
$\r_t\otimes\id_G(w_G)$ is the function $s\mapsto u_{st}=u_su_t$. Thus
$\alpha_t\otimes \id_G\big(\phi\otimes \id_G(w_G)\big)=\phi\otimes \id_G(w_G)(1\otimes u_t)$.
Using this identity, we get
\begin{align*}
\delta_\pn(m)&=
\int_H^{\st} \big(\phi\otimes \id_G(w_G)\big(\alpha_t(\xi\eta^*)\otimes u_t u_t^*)\phi\otimes \id_G(w_G^*)\big)\,dt\\
&= \phi\otimes \id_G(w_G)\left(\int_H^{\st}\alpha_t(\xi\eta^*)\,dt \otimes 1\right)\phi\otimes \id_G(w_G^*)\\
&=\phi\otimes \id_G(w_G)(m\otimes 1)\phi\otimes \id_G(w_G^*)
\end{align*}
We should point out that a similar formula as in~\eqref{eq-deltap}
is given for the reduced case in \cite[Theorem 4.1]{Kaliszewski-Quigg-Raeburn:ProperActionsDuality}.
\end{remark}

We now return to the our original situation where $H=G$ and we use the above lemma in this case to prove the main result of this section:

\begin{proof}[Proof of Theorem \ref{thm-landstad}]
By Lemma \ref{lem-wg} (applied to $H=G$) and the above remark we obtain a coaction $\delta_\pn$ on $B_\pn:=A^{G,\alpha}_\pn$
such that $(\F_\pn(A),\delta_{\F_\pn(A)})$ implements a Morita equivalence between $(B_\pn,\delta_\pn)$ and
$(A\rtimes_{\alpha,\pn}G, \widehat{\alpha}_\pn)$. Hence we get (1).

Statement (2) follows from Lemma \ref{lem-max} together with the fact that $\widehat{\alpha}_u$
is the maximalization of $\widehat{\alpha}_\pn$ and $\widehat{\alpha}_r$ is the normalization of $\widehat{\alpha}_\pn$.

In order to prove (3) we first check that $(k,\phi)$ is a covariant homomorphism of $(B_\pn, G,\delta_\pn)$. In fact, for $m\in A_c^{G,\alpha}$ we have
\begin{align*}
(k\otimes \id_G)\circ \delta_\pn(m)&=
\phi\otimes \id_G(w_G) (k(m)\otimes 1) \phi \otimes \id_G(w_G^*),
\end{align*}
which implies covariance of $(k,\phi)$. We also have
$$k(B_\pn)\phi(\contz(G))\supseteq A_c^{G,\alpha}\phi(\contz(G)),$$
which is dense in $A$ by \cite[Lemma 3.10 (2)]{Quigg:Landstad_duality}, hence
$A=k\rtimes\phi(B_\pn\rtimes_{\delta_\pn}\widehat{G})$.
Note that $k\rtimes\phi:B_\pn\rtimes_{\delta_{\pn}}G\to A$ is $G$\nb-equivariant, since
\begin{align*}
k\rtimes\phi\big(\widehat{\delta}_\pn(s)(j_B(b)j_{\contz(G)}(f))\big)&=
k\rtimes\phi\big(j_B(b)j_{\contz(G)}(\r_s(f)))\big)\\
&=k(b)\phi(\r_s(f))=\alpha_s(k(b)\phi(f))\\
&=\alpha_s\big(k\rtimes\phi(j_B(b)j_{\contz(G)}(f))\big).
\end{align*}
To see that $k\rtimes \phi:B_\pn\rtimes_{\delta_\pn}\widehat{G}\to A$ is an isomorphism we
use the fact that the crossed product by a coaction is always isomorphic to the crossed product by
its normalization. Moreover, $(k,\phi)$ factors through the covariant homomorphism
$(\id, \phi)$ of $(B_r,G, \delta_r)$ and it follows then from the $G$\nb-equivariance
checked above and \cite[Proposition 3.1]{Quigg:Landstad_duality} that
$\id\rtimes\phi: B_r\rtimes_{\delta_r}\widehat{G}\to A$ is an isomorphism.

We finally have to show that $(B_\pn,\delta_\pn)$ satisfies $\pn$-duality. For this we have to show that
the canonical map $\Phi_{B_\pn}: B_\pn\rtimes_{\delta_\pn}\widehat{G}\rtimes_{\widehat\delta_\pn,u}G\to B_\pn\otimes \K(L^2G)$
factors through an isomorphism $B_\pn\rtimes_{\delta_\pn}\widehat{G}\rtimes_{\widehat\delta_\pn,\pn}G\cong B_\pn\otimes \K(L^2G)$.
Since $\delta_u$ is the maximalization of $\delta_\pn$, we have an isomorphism
$$\Phi_{B_u}:B_\pn\rtimes_{\delta_\pn}\widehat{G}\rtimes_{\widehat\delta_\pn,u}G\cong A\rtimes_{\alpha} G\congto B_u\otimes \K(L^2G).$$
Combining the Morita equivalence $B_u\sim_M A\rtimes_{\alpha}G$ with this isomorphism,
we obtain a Morita equivalence $B_u\sim_M B_u\otimes \K(L^2G)$ given by the equivalence
bimodule $\F_\un(A)\otimes_{A\rtimes_{\alpha}G}\big(B_u\otimes\K(L^2G)\big)$. To finish, we need:

\begin{lemma}\label{lem-Mor}
The $B_u-B_u\otimes \K(L^2G)$-equivalence bimodule $\F_\un(A)\otimes_{A\rtimes_{\alpha}G}\big(B_u\otimes\K(L^2G)\big)$
is isomorphic to $B_u\otimes \F(G)$, where we regard $\F(G)$ as a $\C-\K(L^2G)$ equivalence bimodule
with respect to the isomorphism $M\rtimes \rho: \contz(G)\rtimes G\to \K(L^2G)$.
\end{lemma}
\begin{proof}
By Proposition \ref{prop-factor}, we have
\begin{align*}
 \F_\un(A)&\cong \F(G)\otimes_{\contz(G)\rtimes G} (A\rtimes_{\alpha}G)\cong\F(G)\otimes_{\contz(G)\rtimes G} \big(B_u\otimes \K(L^2G)\big),
 \end{align*}
 where in the last isomorphism we replaced
 $A\rtimes_{\alpha}G=(B_u\rtimes_{\delta_u}\widehat{G})\rtimes_{\widehat{\delta}_u,u}G$ by
 the isomorphic algebra $B_u\otimes \K(L^2G)$ via $\Phi_{B_u}$.
We need to understand the left action of $\contz(G)\rtimes G$ on $B_u\otimes \K(L^2G)$.
The left action of $\contz(G)\rtimes G$
on $A\rtimes_{\alpha}G$ is given by the integrated form of the covariant homomorphism
$(\iota_A\circ \phi, \iota_G)$, where $(\iota_A, \iota_G): (A,G)\to \M(A\rtimes_{\alpha}G)$ denote the canonical maps.
Identifying $B_u\rtimes_{\delta_u}\widehat{G}$ with $A$ via $k\rtimes \phi$ as in Theorem \ref{thm-landstad},
the corresponding action of $\contz(G)\rtimes G$ on $(B_u\rtimes_{\delta_u}\widehat{G})\rtimes_{\widehat{\delta}_u,u}G$
is given by the covariant homomorphism $(i_{B_u\rtimes_{\delta_u}\widehat{G}}\circ j_{\contz(G)}, \iota_G)$.
If we compose this with the isomorphism $\Phi_{B_u}=\big((\id_{B_u}\otimes \lambda)\circ \delta_u\rtimes (1_{B_u}\otimes M)\big)\rtimes (1_{B_u}\otimes \rho)$,
we see that the action of $\contz(G)\rtimes G$ on $B_u\otimes \K(L^2G)$ is given by
$1_{B_u}\otimes (M\rtimes \rho)$. Since $M\rtimes \rho:\contz(G)\rtimes G\to \K(L^2G)$ is an isomorphism, we get
$$\F_\un(A)\cong \F(G)\otimes_{\contz(G)\rtimes G} (B_u\otimes \K(L^2G))\cong B_u\otimes \F(G)$$
if we understand $\F(G)$ as a Hilbert $\C-\K(L^2G)$-bimodule via the isomorphism $M\rtimes \rho:\contz(G)\rtimes G\to \K(L^2G)$.
\end{proof}

We can now finish the proof of Theorem \ref{thm-landstad} as follows: it is clear that $B_\pn$ is
 the quotient of $B_u$ corresponding to
the quotient $B_\pn\otimes \K(L^2G)$ of $B_u\otimes \K(L^2G)$ under the Rieffel-correspondence for the equivalence bimodule $B_u\otimes \F(G)$. By the lemma, this module is isomorphic to $\F_\un(A)$ if we identify
$A\rtimes_{\alpha}G$ with $B_u\otimes \K(L^2G)$ via $\Phi_{B_u}$. But the quotient
of $A\rtimes_{\alpha}G$ corresponding to $B_\pn$ with respect to $\F_\un(A)$ is $A\rtimes_{\alpha,\pn}G
= (B_\pn\rtimes_{\delta_\pn} G)\rtimes_{\widehat{\delta}_\pn,\pn}G$ which implies that
$\Phi_{B_u}$ factors through an isomorphism $\Phi_{B_\pn}$ between
$(B_\pn\rtimes_{\delta_\pn} G)\rtimes_{\widehat{\delta}_\pn,\pn}G$ and $B_\pn\otimes \K(L^2G)$.
\end{proof}

As a consequence of our previous results, we see that for a weak $G\rtimes G$\nb-algebra $A$, the Morita equivalence $A^{G,\alpha}_\pn\sim_M A\rtimes_{\alpha,\pn}G$ implemented by $\F_\pn(A)$ is actually a canonical stable isomorphism:

\begin{corollary}
Let $(A,\alpha)$ be a weak $G\rtimes G$\nb-algebra and let $\|\cdot\|_\pn$ be a crossed-product norm on $\contc(G,A)$ for which the dual coaction on $A\rtimes_{\alpha,\un}G$ factors through a dual coaction on $A\rtimes_{\alpha,\pn}G$. Then $A\rtimes_{\alpha,\pn}G\cong A^{G,\alpha}_\pn\otimes\K(L^2G)$.
\end{corollary}
\begin{proof}
By Theorem~\ref{thm-landstad}, we have $(A,\alpha)\cong (A^{G,\alpha}_\pn\rtimes_{\delta_\pn}\dualG,\dual\delta_\pn)$ as $G\rtimes G$\nb-algebras and $(A^{G,\alpha}_\pn,\delta_\pn)$ satisfies $\pn$-duality, so that
\begin{equation*}
A\rtimes_{\alpha,\pn}G\cong A^{G,\alpha}_\pn\rtimes_{\delta_\pn}\dualG\rtimes_{\dual\delta_\pn,\pn}G\cong A^{G,\alpha}_\pn\otimes\K(L^2G).\qedhere
\end{equation*}
\end{proof}

The proof of the final converse statement in Theorem~\ref{thm-landstad} will now be a consequence
of the following variant of Lemma \ref{lem-Mor}. It also shows that the isomorphism
$A\rtimes_{\alpha,\pn}G\cong A^{G,\alpha}_\pn\otimes\K(L^2G)$ of the above lemma
turns the Morita equivalence $A_\pn^G\sim_M A\rtimes_{\alpha,\pn}G$ into the canonical one:

\begin{lemma}\label{lem-Mor1}
Suppose that $\delta:B\to M(B\otimes C^*(G))$ is a $\pn$\nb-coaction for some given
crossed-product norm $\|\cdot\|_\pn$ on $C_c(G, B\rtimes_{\delta}\dualG)$.
Recall that this means that the canonical homomorphism
$$\Phi:B\rtimes_\delta\dualG\rtimes_{\hat\delta}G\to B\otimes \K(L^2G)$$
factors through an isomorphism $\Phi_\pn\colon B\rtimes_\delta\dualG\rtimes_{\hat\delta,\pn}G \congto B\otimes \K(L^2G)$.
Let the crossed product $(B\rtimes_{\delta}\dualG,\dual\delta)$ be equipped with the canonical weak $G\rtimes G$\nb-algebra structure.
Then there is a canonical isomorphism
$$\F_\pn(B\rtimes_{\delta}\dual G)\cong B\otimes \F(G)$$
as Hilbert $B\otimes \K(L^2G)$-modules if we identify $C_0(G)\rtimes G$ with $\K(L^2G)$ via $M\rtimes \rho$
and $B\rtimes_\delta\dualG\rtimes_{\dual\delta,\pn}G$ with $B\otimes \K(L^2G)$ via $\Phi_\pn$. In particular, the left action of $(B\rtimes_\delta\dualG)_c^{G,\dual\delta}$
on this module extends to an isomorphism
$$(B\rtimes_{\delta}\dualG)_\pn^{G,\dual\delta}= \K(\F_\pn(B\rtimes_{\delta}\dualG))\cong B.$$
This isomorphism sends the coaction $\delta_\pn^G$ on $(B\rtimes_{\delta}\dualG)_\pn^{G,\dual\delta}$ (given by Theorem~\ref{thm-landstad})
to the original coaction $\delta$ on $B$.
\end{lemma}
\begin{proof} The proof of the first assertion is word for word as the proof of Lemma \ref{lem-Mor} in case where
$A=B\rtimes_\delta\dualG$ and where we replace $B_u$ by $B$. This gives the chain of isomorphisms
\begin{align*}
\F_\pn(B\rtimes_{\delta}\dual G)&\cong\F_\pn(B\rtimes_{\delta}\dual G)\otimes_{B\rtimes_\delta\dualG\rtimes_{\hat\delta,\pn}G} (B\otimes \K(L^2G))\\
&\cong \F(G)\otimes_{C_0(G)\rtimes G}(B\otimes \K(L^2G))\cong B\otimes \F(G).
\end{align*}
Now, it is well-known (see Equation~(2.3) in \cite{Echterhoff-Kaliszewski-Quigg:Maximal_Coactions}) that the canonical isomorphism $B\rtimes_\delta\dualG\rtimes_{\hat\delta,\pn}G \cong B\otimes \K(L^2G)$ sends the bidual coaction to the coaction on $B\otimes \K(L^2G)$ given by the formula
\begin{equation}\label{eq:CoactionOnBotimesK}
\delta_{B\otimes\K}(x)=(1\otimes\omega_G^*)(\delta\otimes_*\id)(x)(1\otimes\omega_G)\quad\mbox{for all }x\in B\otimes\K(L^2G),
\end{equation}
where $\delta\otimes_*\id\defeq\sigma\circ(\delta\otimes\id)\colon B\otimes\K(L^2G)\to \M(B\otimes\K(L^2G)\otimes C^*(G))$ and $\sigma\colon C^*(G)\otimes \K(L^2G)\to \K(L^2G)\otimes C^*(G)$ is the flip map. By Lemma~\ref{lem-wg}, the coaction $\delta_{\F_\pn(A)}$ on $\F_\pn(B\rtimes_\delta\dualG)$ corresponds to the coaction on $\F(G)\otimes_{C_0(G)\rtimes G}(B\otimes \K(L^2G))$ which is the (balanced) tensor product (as in \cite{Echterhoff-Kaliszewski-Quigg-Raeburn:Categorical}*{Proposition~2.13}) of the coactions $\delta_{\F(G)}(f)=\omega_G(f\otimes 1)$ on $\F(G)$ and the coaction $\delta_{B\otimes\K}$ on $B\otimes \K(L^2G)$.
And it is easy to see that the canonical isomorphism $\F(G)\otimes_{C_0(G)\rtimes G}(B\otimes \K(L^2G))\cong B\otimes \F(G)$ sends this tensor product coaction to the coaction on $B\otimes \F(G)\cong B\otimes\F(G)$ given by the formula
$$\delta_{B\otimes\F(G)}(b\otimes f)\defeq (\delta\otimes_*\id)(b\otimes f)(1\otimes\omega_G)\quad\mbox{for all }f\in \contc(G), b\in B.$$
Finally, a simple computation shows that the induced coaction on $B\cong \K(B\otimes\F(G))$ coincides with the original coaction $\delta$ on $B$.
This proves the last statement of the lemma since the coaction $\delta_\pn^G$ on $(B\rtimes_{\delta}\dualG)_\pn^{G,\dual\delta}$ is, by definition, the coaction induced by $\delta_{\F_\pn(A)}$ on $\K(\F_\pn(A))\cong \K(B\otimes \F(G))$.
\end{proof}

\section{$E$-duality for ideals in $B(G)$}\label{sec-more}
In the previous section we considered arbitrary crossed-product norms $\|\cdot\|_\pn$
on $\contc(G,A)$ such that the corresponding crossed product $A\rtimes_{\alpha,\pn}G$ admits a dual
coaction. In \cite{Kaliszewski-Landstad-Quigg:Exotic} it is shown that if $E$ is a $G$\nb-invariant weak*-closed ideal in
the Fourier-Stieltjes algebra $B(G)$, then $E$ determines a
crossed-product norm $\|\cdot\|_E$ on $\contc(G,A)$ for any action $\alpha:G\to \Aut(A)$ which
 admits a dual coaction $\widehat{\alpha}_E$.
This allows us to consider functorial properties of the $E$\nb-crossed product functor $(A,G,\alpha)\mapsto
A\rtimes_{\alpha,E}G$. Indeed, it is possible to show that the construction $(A,G,\alpha)\mapsto A\rtimes_{\alpha,E}G$ is a functor between suitable categories
and Proposition~\ref{prop-functor} below already indicates some steps in this direction.

Recall that $B(G)$ consists of all functions of the form $s\mapsto \braket{ \pi(s) \xi}{\eta}$ in which $
\pi:G\to \U(H_\pi)$ is a unitary representation of $G$ and $\xi,\eta\in H_\pi$.
It can be identified with the space $\Cst(G)^*$ of continuous linear functionals on
$\Cst(G)$ via $f(x)=\braket{\pi(x)\xi}{\eta}$ if $f(s)= \braket{ \pi(s) \xi}{\eta}$ for all $s\in G$.
The weak*-topology on $B(G)$ is the one coming from this identification.
For any non-zero $G$\nb-invariant weak*-closed ideal $E\subseteq B(G)$ let $I_E={^\perp E}:=\{a\in \Cst(G): f(a)=0\;\mbox{ for all }f\in E\}$.
It is shown in \cite[Lemma 3.1 and Lemma 3.14]{Kaliszewski-Landstad-Quigg:Exotic} that $I_E$ is a closed ideal in $\Cst(G)$
which is contained in the kernel $\ker\lambda$ of the regular representation of $G$,
and
Kaliszewski, Quigg and Landstad define the {\em $E$\nb-group \cstar{}algebra of $G$} as the quotient \cstar{}algebra
$$C_E^*(G):=\Cst(G)/I_E$$
(see \cite[Definition 3.2]{Kaliszewski-Landstad-Quigg:Exotic}).
 Let $q_E: \Cst(G)\to C_E^*(G)$ denote the quotient map. If $\alpha: G\to \Aut(A)$ is an action,
then Kaliszewski, Quigg and Landstad define the {\em $E$\nb-crossed product} $A\rtimes_{\alpha,E}G$
as
$$A\rtimes_{\alpha, E}G=(A\rtimes_{\alpha}G)/J_{\alpha,E}\quad \text{with}
\quad J_{\alpha,E}=\ker(\id\otimes\,q_E)\circ \widehat{\alpha}_u.$$
The $E$\nb-crossed product
$A\rtimes_{\alpha,E}G$ ``lies between'' the maximal and the reduced crossed products and
the coaction $\widehat{\alpha}_u$ on the full crossed product factors
through a coaction $\widehat{\alpha}_E$ on $A\rtimes_{\alpha,E}G$ by \cite[Theorem 6.2]{Kaliszewski-Landstad-Quigg:Exotic}. We also have $\C\rtimes_EG\cong C_E^*(G)$, which follows from the fact that the comultiplication
on $C^*(G)$ factors through a coaction $\delta: C_E^*(G)\to \M(C_E^*(G)\otimes C^*(G))$.

So assume from now on that $E\subseteq B(G)$ is a $G$\nb-invariant weak*-closed nonzero ideal and that
$\delta:B\to \M(B\otimes \Cst(G))$ is any given coaction. By Theorem \ref{thm-landstad}
and Corollary \ref{cor-landstad}, we know that there exists a canonical coaction $\delta_E:B_E\to \M(B_E\otimes \Cst(G))$
which satisfies $E$\nb-duality and which has the same dual system as the original coaction $\delta$, \ie, $(B_E,\delta_E)$ is an $E$\nb-ization for $(B,\delta)$.
We now want to describe the algebra $B_E$ in terms of $E$. At the same time, we give a positive answer
to \cite[Conjecture 6.14]{Kaliszewski-Landstad-Quigg:Exotic}:

\begin{theorem}\label{thm-deltaE}
Let $\delta_u: B_u\to \M(B_u\otimes \Cst(G))$ denote the maximalization of the coaction
$\delta:B\to \M(B\otimes \Cst(G))$. Then $B_E=B_u/J_{\delta,E}$ with
$$J_{\delta,E}=\ker(\id_{B_u}\otimes\, q_E)\circ \delta_u\subseteq B_u.$$
Moreover, the coaction $\delta_u$ factors to give a coaction $\delta_E:B_E\to \M(B_E\otimes \Cst(G))$
which satisfies $E$\nb-duality. In particular, any dual coaction $\widehat{\alpha}_E$ on any
$E$\nb-crossed product $A\rtimes_{\alpha,E}G$ satisfies $E$\nb-duality.
\end{theorem}
\begin{proof}
Let $A=B\rtimes_\delta \widehat{G}$ equipped with the canonical structure of a weakly proper $G\rtimes G$\nb-algebra
and let $(\F_\un(A), \delta_{\F_\un(A)})$ denote the coaction on the $B_u- A\rtimes_{\alpha}G$ equivalence
bimodule $\F_\un(A)$ of Theorem \ref{thm-landstad}, with $\alpha:=\widehat{\delta}$.
The theorem will follow from Theorem \ref{thm-landstad} as soon as we can show that the
ideal $J_{\delta,E}$ in $B_u$ corresponds to the ideal $J_{\alpha,E}$ in $A\rtimes_{\alpha}G$
under the Rieffel-correspondence.
But this follows from the existence  of the bimodule map
$$(\id_{\F_\un(A)}\otimes\,q_E)\circ \delta_{\F_\un(A)}: \F_\un(A)\to \M(\F_\un(A)\otimes C_E^*(G))$$
which is compatible with $(\id_{B_u}\otimes\, q_E)\circ \delta_u$ on the left and
$(\id_{A\rtimes G} \otimes\, q_E)\circ \widehat{\alpha}_u$ on the right hand side of the module.

If we apply this result to a dual coaction $\delta_E=\widehat{\alpha}_E$, we see that all dual coactions
on $E$\nb-crossed products satisfy $E$\nb-duality.
\end{proof}

It might be reasonable to ask whether every coaction $\delta:B\to \M(B\otimes \Cst(G))$ is
one of the coactions $\delta_E$ for some ideal $E$. By the above theorem, this is actually
equivalent to asking whether every dual coaction $\widehat\alpha_\pn$ on some
intermediate crossed product $A\rtimes_{\alpha,\pn}G$ equals $\widehat\alpha_E$ for
some $G$\nb-invariant weak*-closed ideal $E\subseteq B(G)$. We shall see below that this is
not the case. For the proof, we first need:

\begin{proposition}\label{prop-functor}
Let $E$ be a $G$\nb-invariant weak*-closed ideal of $B(G)$. Let $\alpha:G\to \Aut(A)$ and
$\beta:G\to \Aut(B)$ be actions and let $\Theta:A\to \M(B)$ be a $G$\nb-equivariant \Star{}homomorphism.
Then the inclusion $\contc(G,A)\to \contc(G, \M(B))\subseteq \M(B\rtimes_{\beta,E}G)$ extends to a
\textup(unique\textup) \Star{}homomorphism $\Theta\rtimes_EG:A\rtimes_{\alpha,E}G\to \M(B\rtimes_{\beta,E}G)$.
\end{proposition}
\begin{proof}
Let $\Theta\rtimes_uG:A\rtimes_{\alpha}G\to \M(B\rtimes_{\beta,u}G)$ denote the corresponding map
for the universal norms. We need to show that the ideal $J_{\alpha,E}=\ker(\id_{A\rtimes G}\otimes q_E)\circ\widehat\alpha_u$
contains $\ker(\Theta\rtimes_uG)$. But this follows from the commutativity of the diagram
$$
\begin{CD}
A\rtimes_{\alpha}G @>\Theta\rtimes_uG>> \M(B\rtimes_{\beta,u}G)\\
@V(\id\otimes q_E)\circ \widehat\alpha VV  @VV(\id\otimes q_E)\circ \widehat\beta V\\
\M(A\rtimes_{\alpha}G\otimes C_E^*(G))
@>(\Theta\rtimes_uG)\otimes \id_G>> \M(B\rtimes_{\beta,u}G\otimes C_E^*(G))
\end{CD}
$$
where we extended maps to multiplier algebras where necessary, which
is no problem if $\Theta: A\to \M(B)$ is assumed to be nondegenerate.
In case that $\Theta$ is degenerate, we may replace
$\M(A\rtimes_{\alpha}G\otimes C_E^*(G))$ by the subalgebra
$\widetilde{M}(A\rtimes_{\alpha}G\otimes C_E^*(G))$ consisting of all
$m$ which satisfy $m(1\otimes z)\in
 A\rtimes_{\alpha}G\otimes C_E^*(G)$ for all $z\in C_E^*(G)$
 in the lower left corner of the diagram,
on which there always exists a unique extension of $(\Theta\rtimes_uG)\otimes \id_G$.
 \end{proof}

Note that it follows in particular from the above proposition that for all weak-* closed ideals $E\subseteq B(G)$ the morphism
$\C\to \M(A); \lambda\mapsto \lambda 1$ induces a canonical map
$$i_{\Cst_E(G)}: C_E^*(G)=\C\rtimes_EG\to \M(A\rtimes_{\alpha,E}G).$$

\begin{example}[Counter-example to Conjecture~6.12 in \cite{Kaliszewski-Landstad-Quigg:Exotic}]
\label{ex-coaction}
Let $G$ be any locally compact group such that $\Cst(G)$ is not nuclear (\eg any discrete non-amenable group). Then there exists a \cstar{}algebra $A$ such that
$$A\otimes_{\max} \Cst(G)\not\cong A\otimes \Cst(G)\not\cong A\otimes C_r^*(G),$$
where we understand the symbol $\not\cong$ in the sense that the canonical surjective morphisms from left to right
are not injective. Consider the trivial action $\triv$ of $G$ on $A$. Then
$$A\rtimes_{\triv,u}G\cong A\otimes_{\max}\Cst(G)\quad\text{and}\quad A\rtimes_{\triv,r}G\cong A\otimes C_r^*(G),$$
so that the tensor product $A\otimes \Cst(G)$ can be regarded as a $\pn$\nb-crossed product $A\rtimes_{\triv,\pn}G$
for some crossed-product norm $\|\cdot\|_\pn$ lying between $\|\cdot\|_u$ and $\|\cdot\|_r$.
Moreover, $\id_A\otimes\delta_G$ is a coaction on $A\otimes \Cst(G)$ which corresponds to
the dual coaction $\widehat{\triv}_\pn$
under the identification $A\otimes \Cst(G)\cong A\rtimes_{\triv,\pn}G$.

We claim that $\|\cdot\|_\pn$ is not an $E$\nb-norm for {\bf any} $G$\nb-invariant weak*-closed ideal $E\subseteq B(G)$.
Assume to the contrary that $\|\cdot\|_\pn=\|\cdot\|_E$ for some $E$. Since $\|\cdot\|_\pn$ is strictly smaller than
$\|\cdot\|_u$, we then must have that $C_E^*(G)$ is a proper quotient of $\Cst(G)$. By Proposition~\ref{prop-functor}, the
canonical (injective) inclusion $i_{\Cst(G)}\colon \Cst(G)\to \M(A\otimes \Cst(G))$ must factor through a map
$i_{C_E^*(G)}: C_E^*(G)\to \M(A\otimes \Cst(G))$.
But this is impossible since it forces the surjection $\Cst(G)\to \Cst_E(G)$ to be injective.

It follows from Theorem \ref{thm-deltaE} that the coaction on $A\otimes \Cst(G)$ cannot
satisfy $E$\nb-duality for any $E$ (although it satisfies duality for some crossed-product norm). This gives a negative answer to \cite[Conjecture 6.12]{Kaliszewski-Landstad-Quigg:Exotic}.
\end{example}

\section{Functoriality}

Throughout this section, we fix a locally compact group $G$ and a $G$\nb-invariant weak-* closed ideal $\En$ of $B(G)$ containing $A(G)$ and we denote by $\|\cdot\|_\En$ the corresponding crossed-product norm on $\contc(G,A)$ for any $G$\nb-algebra $(A,\alpha)$. We write $A\rtimes_{\alpha,\En}G$ or simply $A\rtimes_\En G$ for the associated crossed product.
Proposition~\ref{prop-functor} shows that such norms have good functorial properties (and it is, in fact, not difficult to see that the $\En$-crossed product construction $A\mapsto A\rtimes_\En G$ yields a functor between suitable categories). We want to show that our constructions are also functorial for such norms. More precisely, we show that the construction $A\mapsto A^G_\En:=A^{G,\alpha}_{\En}$ extends to a functor between appropriate categories of \cstar{}algebras. We will consider both natural categories of \cstar{}algebras with (equivariant) ordinary \Star{}homomorphisms or nondegenerate
{\Star{}homomorphisms into multiplier algebras} as their morphisms. In both cases we obtain functoriality in complete generality (the actions here might be even non-saturated).

\begin{remark}\label{rem-module-hom}Let $\E_i$ be a Hilbert $B_i$-module for $i=1,2$ and let $\phi\colon B_1\to \M(B_2)$ be a (possibly degenerate) \Star{}homomorphism. A \emph{morphism} from $\E_1$ to $\E_2$ compatible with $\phi$ is a linear map $\psi\colon \E_1\to \M(\E_2)\defeq \Lb(B_2,\E_2)$ satisfying $\psi(\xi\cdot b)=\psi(\xi)\circ \phi(b)$ and $\phi(\braket{\xi}{\eta}_{B_1})=\psi(\xi)^*\circ\psi(\eta)$ for all $\xi,\eta\in \E_1$ and $b\in B_1$. Such a morphism induces a \Star{}homomorphism $\tilde\psi\colon \K(\E_1)\to \Lb(\E_2)=\M(\K(\E_2))$ satisfying
$$\tilde\psi(_{\K(\E_1)}\braket{\xi}{\eta})=\psi(\xi)\circ\psi(\eta)^*$$
(see \cite[Lemma~2.2]{Kashiwara-Pinzari-Watatani:Ideal_structure}).
If $\phi$ is injective (\ie, isometric), then so is $\psi$.

Moreover, we say that such $\psi\colon \E_1\to \M(\E_2)$ is {\em nondegenerate} if $\psi(\E_1)B_2$ is dense in $\E_2$.
It follows then that $\tilde\psi:\K(\E_1)\to \M(\K(\E_2))$ is also nondegenerate by \cite[Example 1.10]{Echterhoff-Raeburn:Multipliers}.
\end{remark}

\begin{proposition}\label{prop-functorFA}
Let $A$ and $B$ be two $X\rtimes G$-algebras and let $\pi\colon A\to \M(B)$ be a (possibly degenerate) $X\rtimes G$-equivariant \Star{}homomorphism, that is, a \Star{}homomorphism which intertwines  the actions of $G$ and $\contz(X)$ on $A$ and $B$. Then $\pi$ induces a morphism $\F_\En(\pi)\colon \F_\En(A)\to \M(\F_\En(B))=\Lb(B\rtimes_\En G,\F_\En(B))$ of Hilbert modules which is compatible with the homomorphism
$\pi\rtimes_\En G\colon A\rtimes_\En G\to \M(B\rtimes_\En G)$ (from Proposition~\ref{prop-functor})  and which is given by the formula:
\begin{equation}\label{eq:FormularForMorphismF(pi)}
\F_\En(\pi)(\xi)g=\pi(\xi)*g=\int_G\Delta(t)^{-1/2}\beta_t(\pi(\xi)g(t^{-1}))\dd{t}
\end{equation}
for all $\xi\in \F_c(A)=\contc(X)\cdot A$ and $g\in \contc(G,B)$. Moreover, $\pi$ induces a
\Star{}homo\-morphism $\pi^G_\En\colon A^G_\En\to \M(B^G_\En)$ satisfying $\pi^G_\En(\EE^\alpha(a))=\EE^\beta(\pi(a))$ for all $a\in A_c=\contc(X)\cdot A\cdot \contc(X)$. If $\pi\rtimes_\En G$ is injective or nondegenerate (the latter follows, if $\pi$ is nondegenerate), then so are $\F_\En(\pi)$ and $\pi^G_\En$.
\end{proposition}
\begin{proof}
The \Star{}homomorphism $\pi\rtimes_\En G$ induces a compatible morphism
$$\id_{\F(X)}\otimes\pi\rtimes_\En G\colon\F(X)\otimes_{\contz(X)\rtimes G}(A\rtimes_\En G)\to \M(\F(X)\otimes_{\contz(X)\rtimes G}(B\rtimes_\En G))$$
which maps $f\otimes h\in \contc(X)\otimes \contc(G,A)$ to $f\otimes (\pi\rtimes_\En G)(h)$. Under the canonical isomorphisms $\F(X)\otimes_{\contz(X)\rtimes G}(A\rtimes_\En G)\cong \F_\En(A)$ and $\F(X)\otimes_{\contz(X)\rtimes G}(B\rtimes_\En G)\cong \F_\En(B)$ this therefore induces a compatible morphism $\F_\En(\pi)\colon \F_\En(A)\to \M(\F_\En(B))$
and the description of the above decompositions for $\F_\En(A)$ and $\F_\En(B)$ in (\ref{eq-decom})
together with a straightforward computation
shows that $\F_\En(\pi)$ is given
by \eqref{eq:FormularForMorphismF(pi)}. By the above remark, $\F_\En(\pi)$ induces a \Star{}homomorphism $\pi^G_\En\colon A^G_\En=\K(\F_\En(A))\to \M(\K(\F_\En(B)))=\M(B^G_\En)$ determined by
$\pi^G_\En(_{\K}\braket{\xi}{\eta})=\F_\En(\pi)(\xi)\circ\F_\En(\pi)(\eta)^*$ for all $\xi,\eta\in \F_c(A)$.
Now, the adjoint operator $\F_\En(\pi)(\eta)^*$ is easily seen to be given by
$$\F_\En(\pi)(\eta)^*\zeta|_t=
\bbraket{\pi(\eta)}{\zeta}_{\contc(G,B)}(t)=\Delta(t)^{-1/2}\pi(\eta)^*\beta_t(\zeta)$$
 for all $\zeta\in \F_c(B)$. Hence, for $a=\xi\eta^*\in A_c$, we get
\begin{multline*}
\pi^G(\EE^\alpha(a))\zeta=\pi^G_\En(_{\K}\braket{\xi}{\eta})\zeta=\F_\En(\pi)(\xi)\circ\F_\En(\pi)(\eta)^*\zeta\\
=\pi(\xi)*\bbraket{\pi(\eta)}{\zeta}_{\contc(G,B)}=\EE^\beta(\pi(\xi\eta^*))\zeta=\EE^\beta(\pi(a))\zeta.
\end{multline*}
The last assertion concerning injectivity follows from Remark~\ref{rem-module-hom}.
The assertion about nondegeneracy follows from the fact that if $\pi$ is nondegenerate, then so is $\pi\rtimes_\En G$ and hence also $\id_{\F(X)}\otimes\pi\rtimes_\En G\cong \F_\En(\pi)$ and its induced homomorphism $\pi^G$ on compact operators.
\end{proof}

\begin{remark}\label{rem:Fix-Functor}
{\bf (1)} Given $x\in \contc(X)\cdot \M(B)$ (for a weakly proper $X\rtimes G$-algebra $B$), notice that $x$ may be viewed as an element of $\M(\F_\En(B))$  (that is, a multiplier of the Hilbert module $\F_\En(B)$) by the formula $x\cdot g\defeq \int_G\Delta(t)^{-1/2}\beta_t(xg(t^{-1}))\dd{t}$ (and adjoint given by $x^*\zeta|_t=\Delta(t)^{-1/2}x^*\beta_t(\zeta)$ for all $\zeta\in \F_c(B)$). In this way, the formula~\eqref{eq:FormularForMorphismF(pi)} for the morphism $\F_\En(\pi)$
can be described more easily as $\F_\En(\pi)(\xi)=\pi(\xi)$ for all $\xi\in \F_c(A)$, \ie, the map $\F_\En(\pi)$ is essentially $\pi$ when restricted to $\F_c(A)\subseteq A$.
The homomorphism $\pi^G_\En\colon A^G_\En\to\M(B^G_\En)$ may also be interpreted similarly. In fact, if $\pi$ is nondegenerate, then $\pi^G_\En(\EE^\alpha(a))=\EE^\beta(\pi(a))=\pi(\EE^\alpha(a))$ for all $a\in A_c$, so that $\pi^G_\En$ coincides with the extension of $\pi$ to $\M(A)$ on $A^G_c\subseteq \M(A)$.
In general, for a possibly degenerate \Star{}homomorphism $\pi\colon A\to\M(B)$, we can consider the
bidual von Neumann algebras $A''$ and $B''$ and the weakly continuous extension $\pi''\colon A''\to B''$ of $\pi$.
Then $\EE^\beta(\pi(a))=\pi''(\EE^\alpha(a))$ for all $a\in A_c$ because $\EE^\alpha(a)=\int^\st_G\alpha_t(a)\dd{t}$ may be interpreted as a weak limit in $\M(A)\sbe A''$.
Therefore we can still say that $\pi^G_\En$ is the restriction of $\pi''$ to $A^G_c$.

{\bf (2)} For nondegenerate $G$\nb-equivariant homomorphisms $\pi\colon A\to \M(B)$ and reduced norms (\ie, for $\En=B_\red(G)$), the existence of $\pi^G_\red\colon A^G_\red\to \M(B^G_\red)$ was already obtained in \cite[Proposition~2.6]{Kaliszewski-Quigg-Raeburn:ProperActionsDuality} using a more direct approach (which makes the proof much more technical and involved). The proof in this case could  have been done also by using the above idea and the tensor decomposition $\F_\red(A)\cong \F(X)\otimes_{\contz(X)\rtimes G}(A\rtimes_\red G)$ already available from \cite[Theorem~7.1]{Meyer:Generalized_Fixed}.
\end{remark}

The above result already indicates that the assignment $A\mapsto A^G_\En$ is a functor between two different types of \cstar{}categories. To be more precise, for a fixed locally compact group and a $G$\nb-space $X$, let us consider the categories $\Cstd{X}{G}$ and $\Cstnd{X}{G}$ whose objects (in both) are weak $X\rtimes G$-algebras and whose morphisms $A\to B$ in
$\Cstd{X}{G}$ are $X\rtimes G$-equivariant \Star{}homomorphisms from $A$ into $B$, and in $\Cstnd{X}{G}$ are $X\rtimes G$-equivariant \emph{nondegenerate} \Star{}homomorphism
$A\to \M(B)$. In particular, considering the trivial group $\{e\}$ and a space $Y$, we may consider the categories $\CCstd{Y}\defeq \Cstd{Y}{\{e\}}$ and
$\CCstnd{Y}{}\defeq \Cstd{Y}{\{e\}}$ consisting of weak $Y$-algebras (\ie, \cstar{}algebras $A$ endowed with a nondegenerate \Star{}homomorphism $\contz(Y)\to \M(A)$) and whose
morphisms are either $Y$-equivariant \Star{}homomorphisms $A\to B$ or nondegenerate $Y$-equivariant \Star{}homomorphisms $A\to \M(B)$. In the special case that $Y=\{\pt\}$ is the one-point space, we get the ordinary categories $\CCCstd$ and $\CCCstnd$ of \cstar{}algebras with \Star{}homomorphisms or nondegenerate \Star{}homomorphisms as their morphisms.

\begin{corollary}
If $G$ acts properly on a space $X$ and $\|\cdot \|_\En$ denotes a \cstar{}norm coming from a nonzero $G$\nb-invariant
weak-$^*$-closed
ideal $\En\sbe B(G)$,
then the assignments $A\mapsto A^G_\En$ and $\pi\mapsto \pi^G_\En$ constructed above yield functors $\Cstd{X}{G}\to \CCstd{G\bs X}$ and $\Cstnd{X}{G}\to \CCstnd{G\bs X}$.
\end{corollary}
\begin{proof}
For a $X\rtimes G$-equivariant \Star{}homomorphism $\pi\colon A\to B$ it is clear that $\F_\En(\pi)$ maps $\F(A)$ into $\F(B)$ and hence that $\pi^G$ maps $A^G_\En$ into $B^G_\En$.
It is also clear that these maps are $G\bs X$-equivariant. An easy exercise shows that $\pi\mapsto \F_\En(\pi)$ (and hence also $\pi\mapsto \pi^G_\En$) respects composition, \ie, $\F_\En(\pi\circ\rho)=\F_\En(\pi)\circ \F_\En(\rho)$. This yields the functor $\Cstd{X}{G}\to \CCstd{G\bs X}$ and the functor $\Cstnd{X}{G}\to \CCstnd{G\bs X}$ is obtained similarly.
\end{proof}

\section{Categorical Landstad Duality}

In this section we interpret our main result on Landstad duality (Theorem~\ref{thm-landstad}) in categorical terms extending to exotic norms one of the main results by Kaliszewski, Quigg and Raeburn in \cite{Kaliszewski-Quigg-Raeburn:ProperActionsDuality} for reduced generalized fixed-point algebras.

For a fixed locally compact group $G$, we let $G$ act on it self by right translation and consider the categories $\Cstd{G}{G}$ and $\Cstnd{G}{G}$
(already considered in the previous section). Both categories have the same objects, namely, weak $G\rtimes G$\nb-algebras, the only difference between them are their morphisms
which are either $G\rtimes G$\nb-equivariant \Star{}homomorphisms or nondegenerate homomorphisms into multiplier algebras.

Dually, we consider the categories $\CCstd{\dualG}$ and $\CCstnd{\dualG}$ whose objects are (in both) \emph{$\dualG$\nb-algebras}, that is, pairs $(B,\delta)$ consisting of a \cstar{}algebra $B$ and a $G$\nb-coaction $\delta$ on it. The morphisms in $\CCstd{\dualG}$ are $\dualG$\nb-equivariant \Star{}homomorphisms and in $\CCstnd{\dualG}$ are nondegenerate
$\dualG$\nb-equivariant \Star{}homomorphisms into multiplier algebras, where $\dualG$\nb-equivariance of a \Star{}homomorphism between two $\dualG$\nb-algebras means that it commutes with the underlying coactions
in the usual sense (see \cite{Echterhoff-Kaliszewski-Quigg-Raeburn:Categorical} for details).

As observed before, for a given $\dualG$\nb-algebra, the crossed product $B\rtimes_\delta \dualG$ carries a
canonical structure as a $G\rtimes G$\nb-algebra given by the dual $G$\nb-action and the canonical $\contz(G)$-embedding $j_{\contz(G)}$ into $\M(B\rtimes_\delta \dualG)$.
Moreover, it is well-known (see \cite{Echterhoff-Kaliszewski-Quigg-Raeburn:Categorical} for further details)
that the crossed product construction $(B,\delta)\mapsto B\rtimes_\delta \dualG$ may be viewed as a
functor $\CP_{\dualG}$ between the categories $\CCstd{\dualG}\to \Cstd{G}{G}$ and $\CCstnd{\dualG}\to \Cstnd{G}{G}$.

From now on, we fix a crossed-product norm $\|\cdot\|_\En$ associated to a nonzero $G$\nb-invariant \emph{ideal} $E\sbe B(G)$.
For such an $\En$, we are interested in the full subcategories $\CCstd{\dualG}_\En$ and $\CCstnd{\dualG}_\En$ of $\CCstd{\dualG}$ and $\CCstnd{\dualG}$, whose objects are $\En$-coactions as defined in  Definition \ref{def-E-dual}.

Given a weak $G\rtimes G$\nb-algebra $(A,\alpha)$, Theorem~\ref{thm-landstad} implies that the $\En$-generalized fixed-point algebra $A^G_\En$ carries an $\En$-coaction $\delta_\En$, that is, it may be viewed as an object of $\CCstd{\dualG}_\En$ or $\CCstnd{\dualG}_\En$. Moreover, we already know that $\Fix^G_\En$ may be viewed as a functor $\Cstd{G}{G}\to \CCCstd$ or $\Cstnd{G}{G}\to \CCCstnd$. We want to view $\Fix^G_\pn$ as a functor $\Cstd{G}{G}\congto \CCstd{\dualG}_\En$ or $\Cstnd{G}{G}\congto \CCstnd{\dualG}_\En$.
For this the only missing point is to see that $\Fix^G_\En$ is well-defined on the level of morphisms, \ie, if $\pi$ is a morphism from $A$ to $B$ in $\Cstd{G}{G}$ or $\Cstnd{G}{G}$, then the induced morphism $\pi^G_\En$ from $A^G_\En$ to $B^G_\En$ (in $\CCCstd$ or $\CCCstnd$) is $\dualG$\nb-equivariant and hence a morphism in
$\CCstd{\dualG}_\En$ or $\CCstnd{\dualG}_\En$. This will follow from the following lemma:

\begin{lemma}
If $A$ and $B$ are weak $G\rtimes G$\nb-algebras and $\pi\colon A\to\M(B)$ is a $G\rtimes G$\nb-equivariant \Star{}homomorphism,
then the induced \Star{}homomorphism $\pi^G_\En\colon A^G_\En\to\M(B^G_\En)$ is $\dualG$\nb-equivariant.
\end{lemma}
\begin{proof}
By construction (see the proof of Proposition~\ref{prop-functorFA}), the homomorphism \linebreak
$\pi^G_\En\colon A^G_\En\to \M(B^G_\En)$ is  induced on the algebras of compact operators by the morphism $\tilde\pi\defeq \F^G_\En(\pi)\colon \F^G_\En(A)\to\M(\F^G_\En(B))$ given as in Proposition~\ref{prop-functorFA}, and the coactions on $A^G_\En$ and $B^G_\En$ are induced by the corresponding $G$\nb-coactions $\delta_{\F^G_\En(A)}$ on $\F^G_\En(A)$ and $\delta_{\F^G_\En(B)}$ on $\F^G_\En(B)$ as defined in Lemma~\ref{lem-wg}.
Hence, it is enough to see that $\F(\pi)$ is $\dualG$\nb-equivariant with respect to the $G$\nb-coactions $\delta_{\F^G_\En(A)}$ and $\delta_{\F^G_\En(B)}$, \ie, it is enough to show that
$$(\tilde\pi\otimes\id)\big(\delta_{\F^G_\En(A)}(\xi)\big)=\delta_{\F^G_\En(B)}(\tilde\pi(\xi))\quad\mbox{for all }\xi\in \F^G_c(A).$$
For this we use the formula $\delta_{\F^G_\En(A)}(\xi)=(\phi_A\otimes\id)(\omega_G)(\xi\otimes 1)$ for the coaction $\delta_{\F^G_\En(A)}$, where $\phi_A$ denotes the structural map $\contz(G)\to\M(A)$. As in the proof of Lemma~\ref{lem-wg}, the element $(\phi_A\otimes\id)(\omega_G)(\xi\otimes 1)$ is (rigorously) in $\M(A\otimes C^*(G))$,
but is actually interpreted as an element in $\M(\F^G_\En(A)\otimes C^*(G))$. Similarly, we have the formula $\delta_{\F^G_\En(B)}(\eta)=(\phi_B\otimes\id)(\omega_G)(\eta\otimes 1)$ for all $\eta\in \F^G_c(B)$. Moreover, this same formula (with a similar interpretation) also holds for $\eta\in \phi_B(\contc(G))\M(B)$, so that we may also apply it for $\eta=\tilde\pi(\xi)=\pi(\xi)$ (using the interpretation for $\tilde\pi=\F^G_\En(\pi)$ given in Remark~\ref{rem:Fix-Functor}).
Using this and the equivariance of $\pi$ with respect to the $\contz(G)$-homomorphisms $\phi_A$ and $\phi_B$ (that is, the fact that $\pi(\phi_A(f) a)=\phi_B(f)\pi(a)$ for all $a\in A$), we get, for every element $z\in C^*(G)$,
\begin{align*}
(\tilde\pi\otimes\id)\big(\delta_{\F^G_\En(A)}(\xi)\big)(1\otimes z)&=(\pi\otimes\id)\big(\phi_A\otimes\id(\omega_G)(\xi\otimes z)\big)\\
&=(\phi_B\otimes\id)(\omega_G)(\pi(\xi)\otimes z)=\delta_{\F^G_\En(B)}(\tilde\pi(\xi))(1\otimes z).
\end{align*}
Since $z\in C^*(G)$ was arbitrary, this yields the desired $\dualG$\nb-equivariance of $\tilde\pi$.
\end{proof}

The above lemma implies that the assignment $\Fix^G_\En$ which maps $A$ to $A^G_\En$ and a morphism $\pi$ from $A$ to $B$ to $\pi^G_\En$
may be viewed as a functor $\Cstd{G}{G}\to \CCstd{\dualG}_\En$ or $\Cstnd{G}{G}\to \CCstnd{\dualG}_\En$.
The following result extends to exotic generalized fixed-point algebras some results which have been obtained in \cite{Kaliszewski-Quigg-Raeburn:ProperActionsDuality}*{Theorem~4.2} in the case of reduced generalized fixed-point algebras:

\begin{theorem}[Categorical Landstad Duality]
For a locally compact group $G$ and a crossed-product norm $\|\cdot\|_\En$ associated to a nonzero $G$\nb-invariant ideal $\En\sbe B(G)$,
the functor $\Fix_\En^G\colon (A,\alpha)\mapsto (A^G_\En,\delta_\En)$ is an equivalence between the categories $\Cstd{G}{G}\congto \CCstd{\dualG}_\En$ and $\Cstnd{G}{G}\congto \CCstnd{\dualG}_\En$.
The crossed product functor $\CP_{\dualG}\colon (B,\delta)\mapsto B\rtimes_\delta\dualG$ is a quasi-inverse functor of $\Fix^G_\En$.
\end{theorem}
\begin{proof}
Given a weak $G\rtimes G$\nb-algebra $(A,\alpha)$, Theorem~\ref{thm-landstad} implies that the $\En$-general\-ized fixed-point algebra $A^G_\En$ carries an $\En$-coaction $\delta_\En$ of $G$
and there is a canonical $G\rtimes G$\nb-equivariant isomorphism $A^G_\En\rtimes_{\delta_\En} \dualG\congto A$ which is
given as the integrated form $\kappa_A\rtimes\phi_A$
of the covariant representation $(\kappa_A,\phi_A)\colon (A^G_\En,\contz(G))\to \M(A)$, where $\kappa_A\colon A^G_\En\to \M(A)$
is the extension of the inclusion map $A^G_c\into \M(A)$ (\ie, the representation given in Proposition~\ref{prop:RegRepInduced}) and $\phi_A\colon \contz(G)\to \M(A)$
denotes the structural $\contz(G)$-homomorphism of the $G\rtimes G$\nb-algebra $A$. Moreover, Theorem~\ref{thm-landstad} also says that given an $\En$-coaction $(B,\delta)$,
there is an isomorphism $(A^G_\En,\delta_\En)\cong (B,\delta)$. This means that the functor $\Fix_\En^G$ is essentially surjective (between each of the above pairs of categories).
To prove that $\Fix_\En^G$ is an equivalence, it is enough to show that it is full and faithful, \ie, for each pair of weak $G\rtimes G$\nb-algebras $(A,\alpha)$ and $(B,\beta)$,
$G\rtimes G$\nb-equivariant morphisms $(A,\alpha)\to (B,\beta)$ correspond bijectively to $\dualG$\nb-equivariant morphisms $(A^G_\En,\delta_\En^A)\to (B^G_\En,\delta_\En^B)$
via the maps on morphisms induced by the functors $\Fix_\En^G$ and $\CP_{\dualG}$, \ie, the map $\Mor(A,B)\to \Mor(A^G_\En,B^G_\En)$ sending $\pi$ to $\pi^G_\En$ is a bijection,
where $\Mor(x,y)$ denotes the set of morphisms $x\to y$ between objects of the underlying category (which in our case can be any one of the categories appearing in the statement).
We prove this at the same time for both pairs of categories involving ordinary (equivariant) \Star{}homomorphisms or nondegenerate \Star{}homomorphisms into multiplier algebras as their morphisms. For this let $\pi\colon A\to \M(B)$ be a (possibly degenerate) $G\rtimes G$\nb-equivariant \Star{}homomorphism. The induced homomorphism $\pi^G_\En\colon A^G_\En\to \M(B^G_\En)$ is given by $\pi^G_\En(\EE_A(a))=\EE_B(\pi(a))$ for all $a\in A_c$ (see Proposition~\ref{prop-functorFA}), where $\EE_A(a)=\int_G^\st\alpha_t(a)dt$ and similarly for $\EE_B$. We now prove that the following diagram commutes:
$$
\begin{CD}
A^G_\En\rtimes\dualG @>\kappa_A\rtimes\phi_A>> A\\
@V\pi^G_\En\rtimes\dualG VV  @VV\pi V\\
\M(B^G_\En\rtimes\dualG)
@>\kappa_B\rtimes\phi_B>> \M(B)
\end{CD}
$$
In fact, this follows from the following computation: for all $a\in A_c$ and $f\in \contc(X)$ we get
\begin{align*}
\pi((\kappa_A\rtimes\phi_A)(j_{A^G_\En}(\EE_{A}(a))j_G^{A^G_\En}(f))&=\pi(\EE_A(a)\phi_A(f))=\EE_B(\pi(a))\phi_B(f)\\
    &=(\kappa_B\rtimes\phi_B)(j_{B^G_\En}(\EE_B(\pi(a)))j_G^{B^G_\En}(f))\\
    &=(\kappa_B\rtimes\phi_B)(\pi^G_\En\rtimes\dualG)(j_{A^G_\En}(\EE(a))j_G^{A^G_\En}(f)),
\end{align*}
where $j_G$ is an abbreviation for $j_{\contz(G)}$.
The commutativity of the above diagram means that $\pi\colon A\to \M(B)$ is determined by $\pi^G_\En$ via the canonical isomorphisms
$\kappa_A\rtimes\phi_A\colon A^G_\En\rtimes\dualG\congto A$ and $\kappa_B\rtimes\phi_B\colon B^G_\En\rtimes\dualG\congto B$. More precisely, this means that the canonical map
$\Mor(A,B)\to \Mor(A^G_\En,B^G_\En)$ which sends $\pi$ to $\pi^G_\En$ is injective.
It is also surjective because given a (possibly degenerate) $\dualG$\nb-equivariant \Star{}homomorphism $\sigma\colon A^G_\En\to \M(B^G_\En)$, we can consider the corresponding
\Star{}homomorphism $\sigma\rtimes\dualG\colon A^G_\En\rtimes\dualG\to \M(B^G_\En\rtimes\dualG)$ which is $G\rtimes G$\nb-equivariant. Composing $\sigma\rtimes\dualG$ with the canonical
isomorphisms $\kappa_A\rtimes\phi_A$ and $(\kappa_B\rtimes\phi_B)^{-1}$, this yields a $G\rtimes G$\nb-equivariant homomorphism
$\pi\defeq (\kappa_B\rtimes\phi_B)^{-1}\circ(\sigma\rtimes\dualG)\circ \kappa_A\rtimes\phi_A$ such that $\pi^G_\En=\sigma$ (by the above commutative diagram).
Therefore $\Fix^G_\En$ yields an equivalence of categories $\Cstd{G}{G}\congto \CCstd{\dualG}_\En$ and $\Cstnd{G}{G}\congto \CCstnd{\dualG}_\En$. Moreover, our arguments above show that
the composition functor $\CP_{\dualG}\circ\Fix^G_\En$ is naturally isomorphic to the identity functor on $\Cstd{G}{G}$ or $\Cstnd{G}{G}$: the natural isomorphism at a given object (\ie, a weak $G\rtimes G$\nb-algebra) $A$ is the canonical isomorphism $\kappa_A\rtimes\phi_A\colon A^G_\En\rtimes\dualG\congto A$ (notice that the arguments above show that this is natural).
This means that $\CP_{\dualG}$ is a left quasi-inverse for $\Fix^G_\En$. Since we already know that $\Fix^G_\En$ is an equivalence,
it follows that $\CP_{\dualG}$ is a quasi-inverse for $\Fix^G_\En$ (and hence also an equivalence functor).
\end{proof}

\begin{remark}
In this paper we did not consider functoriality of our generalized fixed-point algebra constructions for (equivariant) categories based on correspondences (as done in \cite{Huef-Raeburn-Williams:FunctorialityGFPA} for the reduced case), but it is indeed possible to obtain functoriality also in this setting (and we postpone the full proof of this fact to a forthcoming paper) under the assumption that all actions involved are saturated (which includes the important case where the action on the underlying proper $G$\nb-space $X$ is free -- this is the assumption made in \cite{Huef-Raeburn-Williams:FunctorialityGFPA} -- and, in particular, covers the case of $G\rtimes G$\nb-algebras). We are not going to discuss this, but it is certainly also possible to prove a version of the categorical Landstad duality theorem above for categories of weak $G\rtimes G$\nb- and $\dualG$\nb-algebras with equivariant correspondences as their morphisms.
\end{remark}

\vskip 0,5pc

\end{document}